\documentclass[11pt,oneside]{article}

\usepackage{euscript}
\usepackage{amssymb}
\usepackage{amsthm}
\usepackage{amsmath}
\usepackage{bbm}

\newtheorem{proposition}{{\bf Proposition}}%[section]
\newtheorem{lemma}{{\bf Lemma}}%[section]
\newtheorem{corollary}{{\bf Corollary}}%[section]
\newtheorem{assumption}{{\bf Assumptions}}%[section]
\newtheorem{theorem}{{\bf Theorem}}% [section]
\theoremstyle{remark}
\newtheorem{example}{{\bf Example}}%[section]
\newtheorem{definition}{{\bf Definition}}%[section]
\newtheorem{remark}{{\bf Remark}}%[section]
\def\N{\mathbb N}
\def\R{\mathbb R}
\def\Z{\mathbb Z}
\def\E{\mathbb E}
\def\B{\mathcal B}
\def\F{\mathcal F}
\def\C{\mathbb C}
\def\DS{\mathbb D}

\def\x{\boldsymbol{x}}
\def\y{\boldsymbol{y}}
\def\t{\boldsymbol{t}}

\def\al{\alpha}
\def\va{\varphi}
\def\var{\varepsilon}
\def\la{\lambda}

\def\1{\mathbb 1}

\def\1{\mathbbm{1}}

\begin{document}

\vskip5cm

\centerline{\textbf{ RANDOMIZED  LIMIT THEOREMS }}
\vskip0.2cm

\centerline{\textbf{  FOR  STATIONARY ERGODIC  RANDOM PROCESSES AND FIELDS }}

\vskip 1cm

\centerline{Youri Davydov}   %{\fnms{} \snm{}\ead[label=e1]{axt12@psu.edu}}
\vskip 0.2 cm
\noindent
{\it Faculty of Mathematics and Computer Sciences of St. Petersburg  State University, Russia, and Department of Mathematics of Lille University, France.}

\vskip 0.5 cm

\centerline{ Arkady  Tempelman}
\vskip 0.2 cm
\noindent
{\it Department of Mathematics and
Department of Statistics,
the Pennsylvania State University,}
 USA.

\vskip0.8cm
\begin{abstract}
We consider "randomized" statistics constructed  by using a finite  number of   observations  a  random field at  randomly chosen points. We generalize  the invariance principle (the functional CLT), the Glivenko--Cantelli theorem,  the theorem about convergence to the Brownian bridge and the Kolmogorov   theorem  about the limit distribution of the empirical distribution function,  as well as an improved version of the CLT in     A. Tempelman, Randomized multivariate  central limit theorems for  ergodic homogeneous  random fields,
Stochastic Processes and their Applications.
 143 (2022), 89-105.  The randomized approach, introduced in the mentioned  work,  allows to extend these theorems to  all ergodic  homogeneous random fields on $\Z^m$ and $\R^m.$

\end{abstract}

\vskip 2cm



\newpage

 \section{Preliminaries}\label{P}
\subsection{Short review of the article}
The article is devoted to the extension of  the main Probability limit theorems   to ergodic stationary random processes and to ergodic homogeneous random fields. Our approach is based on consideration of "randomized" statistics, i.e. statistics  constructed  by using a finite  number of   observations of  a  random field at  randomly chosen points, introduced in\cite{T1}.

In subsection \ref{randomizing} and \S \ref{randomization} we define finite  randomizing sets in the  "time" space and randomized statistics, calculated  on restrictions of the fields to these sets.

In \S \ref{Lindeberg} the fulfillment   of a "randomized" form of the  Lindeberg condition is proved for ergodic homogeneous random fields, possessing the second moment; it is essentially used in the sequel.

In \S \ref{generalCLT} we consider  two kinds  of the CLT. Subsection  \ref{generCLT} is devoted to "randomized" versions  of the classical CLTs  which are valid for all ergodic homogeneous fields; in these  theorems the condition $E[|X(0)|^2]<\infty $ replaces the condition: $E[|X(0)|^{2+\delta}]<\infty $ for some  $\delta$, imposed by the second author in   \cite{T1}.
The "randomized" functional central limit theorem (invariance principle) is considered in subsection \ref{IP}.

In \S\ref{ED} two  theorems related to the limiting behavior of the randomized empirical distribution functions (EDF) are considered: a randomized version of the Glivenko--Cantelli theorem for multivariate distributions of  ergodic homogeneous random fields and a general theorem devoted   the limit of  the distributions of the EDF; as a corollary, the randomized version of the Kolmogorov convergence theorem for empirical distributions is derived (of course, these theorems are also valid for  ergodic homogeneous random fields).

\subsection{The "time" set  $T$}\label{time} In this paper we study random fields defined on a set   $T$, which is   the $m$-dimensional Euclidean space $\R^m$ or the $m$-dimensional  integer lattice $\Z^m,\; m\ge 1 $  (when $m=1$, the random fields turn into random processes or random sequences). We denote by $\mathcal B$ the Borel $\sigma$-field on $T$ (if $T=\Z^m$, then $\B$ coincides with  the collection of all subsets of $\Z^m$, and each function $f$ on  $\Z^m$ is $\B$-measurable);  $\la$  is the Lebesgue measure on $\R^m$ and the counting measure on $\Z^m$ (in the latter  case $\la(A)$ is the cardinality of $A\subset\Z^m$,  and $\int_Af(t)\la(dt)= \sum_{t\in A}f(t)$ if $\la(A)<\infty$).

\subsection{Random fields}\label{fields}

We consider a   $d$-dimensional  random field   $X(t)=(X^1(t),...,X^d(t)),\,t\in T, $ over  a probability space $(\Omega_X,\F_X,P_{X})$.

Let us recall several definitions.  The  field $X$ is (strict-sense)\emph{ homogeneous} if all  finite dimensional distributions of $X$ are shift-invariant:
$$P_X\{\omega:X(t_1+t,\omega)\in A_1,...,X(t_k+t,\omega)\in A_k\}=$$
$$P_X\{\omega: X(t_1,\omega)\in A_1,...,X(t_k,\omega)\in A_k\}$$

  \noindent for all  $t,t_i\in T , \;\;i, k\in \N$ and for all sets $A_i$,  belonging to  the Borel  $\sigma$-field  $\B(\R^d)$.

   A family of invertible transformations $\gamma=\{\gamma_t,t\in T\}$ of $\Omega_X$  is said be a \emph{group}   if $\gamma_0\omega\equiv \omega, \gamma_{s+t}= \gamma_s\gamma_t, \gamma^{-1}_t=\gamma_{-t}$, where $s,t\in T,\, \omega\in\Omega_X$. The field  $X$ \emph{is generated by a group} $\gamma$,  if $X(t,\omega)=X(0,\gamma_t\omega),\,t\in T,\,\omega\in\Omega_X$; since $X(0,\gamma_{s+t}\omega)=X(0,\gamma_ s\gamma_t\omega)=X(s,\gamma_t\omega)$, this implies:
   \begin{equation}\label{covariant}
     X(s+t,\omega)=X(s,\gamma_t\omega),s,t\in T.
   \end{equation}
   A family $\gamma$ is \emph{measure preserving}, if the transformations $\gamma_t$ are $\F_X$-measurable and $P_X(\gamma_t\Lambda)=P_X(\Lambda),t\in T,\Lambda\in\F_X$. If     the random field $X$   is ge\-ne\-rated by  a measure preserving group of  transformations   of  $\Omega_X$,   then, by Property \eqref{covariant}, it is   homogeneous. Denote by $\mathcal{I}_X$ the $\sigma$-field of all events $\Lambda \in \F_X$,  which are invariant mod $P_X$ with respect to all transformations $\gamma_t$, i.e. $P_X(\Lambda\triangle\gamma_t\Lambda)=0, t\in T$. The group $\gamma$ is  said to be  \emph{metrically transitive } if $\mathcal{I}_X= \{\Lambda: \Lambda\in\F_X, P_X(\Lambda)=0 \ {\rm or}\ 1\}$.  The field is\emph{ ergodic } if it is   generated by  a metrically transitive measure preserving group $\gamma$.

Without loss of generality, we   assume that  $(\Omega_X,\F_X,P_{X})$ is the probability space of function type, i.e.,  $\Omega_X$ is the space of $\R^d$-valued "sample functions"  $x(\cdot)$ on $T$,  $\F_X$ is the $\sigma$-field generated by the sets \\ $\{x(\cdot):x(t)\in A\},\; t\in T, A\in\B(\R^d),$  and $P_{X}$ is  a probability measure on $\F_X$.  $X$ is the coordinate random field: $X(t,x(\cdot))=x(t),$ where $t\in T,\, x(\cdot)\in \Omega_X$ (sometimes, we  w rite $\omega$ instead of $x(\cdot)$, when we refer to an element of $\Omega_X$).

Consider the     invertible measurable  "shift" transformations $\gamma_t,\;t\in T,$ of  $\Omega_X,$  defined as follows  : $\gamma _t x(\cdot)=x(\cdot+t)$; it is clear that the family $\{\gamma_t\}$ is a group; moreover, $X$ is generated by this group: $$
X(0,\gamma_tx(\cdot))=X(0, x(\cdot+t))=x(0+t)=x(t)=X(t,\cdot)).
$$
The field  is homogeneous, if and only if the shift transformations preserve the measure $P_X$.

Each    component  field $X^l$   may be  considered over   the probability space   $(\Omega_X^l,\F_X^l,P_X^l)$ where $\Omega_X^l$ is the set of scalar functions $x^l(\cdot)$ on $T$, $\F_X^l$ is the $\sigma$-field generated by the events $\{x^l(\cdot):x^l(t)\in D\},\; t\in T,\; D\in\B(\R)$. Let  $A^l$ be an event in $\F_X^l$; the event  $\Lambda_{A^l }:=\{x(\cdot):x^l(\cdot)\in A^l\}\in \F_X$, and  $P_X^l$ is the projection of the measure $P_X$ onto  $\F^l$: for each $A^l\in  \F_X^l , \  P_X^l(A^l) = P_X(\Lambda_{A^l })$.   $X^l$ is the coordinate field: $ X^l(t,x^l(\cdot))=x^l(t),$ where $t\in T,$
$ x^l(\cdot)\in \Omega_X^l$.

 In what follows \emph{we assume that the fields $X^l$ are  homogeneous}, i.e.,  for  each $l$ all transformations  $\gamma_t^l$ preserve the measure $P_X^l$.  It is also assumed  that $ E_X[|X^{l}(0)|]<\infty$    ($ l=1,...,d$) then we set
 $
 \mu^l:=E_X[X^l(0)].$  If $E_X[|X^{l}(0)|^2]<\infty,$ then the variances are denote by $ \sigma^l:= Var_X[X^l(0)].$
\vspace{5pt}

  \emph{If $T=\R^m$, we always assume that the random fields $X^l(t)=X^l(t, x^l)$   are $ \B\times\F^l$-measurable}; by the Fubini theorem, these assumptions imply that \emph{  the sample functions $x^l(\cdot)$ are Borel measurable with $P_X^l$-probability 1 and for each finite Borel measure   $Q$   the integral $\int_Tx^l(t)Q(dt)$ exists with $P_X^l$-probability 1}.

Denote by $\mathcal{I}^l_X$ the $\sigma$-field of all $\gamma$-invariant $\mod(P^l_{X})$ events in $\F_X^l$ (we remind that $X^l$ is said to be ergodic if $\mathcal{I}^l_X$ is trivial).
  $ E_X[X^l(0)|\mathcal{I}^l_X],\linebreak Var_X[X^l(0)|\mathcal{I}^l_X]$ are  the conditional expectation and variance; \emph{in the se\-quel  it is   assumed  that for each $l$ \ $Var_X[X^l(0)|\mathcal{I}_X]>0 $ with $ P^l_{X}$-probability 1}.

  In some cases it is assumed that the $\R^d$-valued random field $X$ is   homoge\-ne\-ous, that is   the transformations $\gamma_t$ preserve the measure $P_X$; then  the transformations $\gamma_t^l$ preserve the measures $P_X^l$, and the components $X^l$ are also  homogeneous;  if the random  field $X$ is ergodic, then all its component fields $X^l$ are ergodic, too.

\subsection{Randomizing random vectors }\label{randomizing} Let $\{q_n^l\}$ be    sequences of    probability measures  on $
\B,\ (l=1,...,d),\; d\in\N$ (these sequences   may coincide for some or even all $l$). For each natural $n$ we consider $d$ mutually independent $k_n$-dimensional random vectors\\
  $\tau_n^l=(\tau^l_{n, 1},...,\tau^l_{n,k_n}),\;$ $l=1,...,d,$ over a probability space $(\Omega_\tau, \F_\tau,P_\tau)$ possessing  the following properties:

  a) the vectors $\tau_n^l$ do not depend on the random field $X,$

\noindent
   and

  b) the components  of each vector $\tau^l$  are i.i.d. $T$-valued random vectors  with the distribution  $q_n^l.$

\noindent
   (Of course, if  $T=\Z^m$ and if for some $l$ the support of $q_n^l$ is finite, then  some points may appear in the sample $\tau^l_{n, 1},...,\tau^l_{n,k_n}$ several times). It is assumed that $k_n\uparrow\infty$ as $n\to\infty$ (the sequence $\{k_n\}$ may also depend on $l$; to simplify the notation,  we always  drop the index  $l$ in $k_n^l$).
Since each scalar field\\  $X^l=X^l(t,\omega) $ is a $\B\times \F_X$-measurable function, $X(\tau^l_{n,i},\omega) $ are random variables over    the probability space $$(\Omega_{X,\tau},\F_{X,\tau},P_{X,\tau}) :=(\Omega_X\times \Omega_\tau,\F_{X}\times \F_\tau,P_{X}\times P_\tau)$$.

\subsection{ Notation}\label{notation} In the sequel  we also use the following notation:

$\bullet$\;\;	
$E_X$, $E_\tau $ and $E_{X,\tau} $ denote the expectation with  respect to the measure $P_X$, $P_\tau$ and $P_{X,\tau}.$

$\bullet$\;\; $Y_n\to Y$ $P$-a.s., $Y_n\overset P\to Y $
 mean  convergence almost sure, respectively  in probability.

$\bullet$\;\;$Y_n\overset {Q}\Longrightarrow Y $ means  convergence  in distribution with respect to the measure $Q.$

$\bullet$\;\; $a:=b$  means  that the quantity $a$ is defined by the expression $b$.

$\bullet$\;\;  $\R,\Z,\N$ denote the sets of real numbers, of integers and of natural numbers.

$\bullet$\;\;  $T=\R^m$ or $\Z^m \ (m\in \N)$.

$\bullet$\;\;$\1_{A}$ or $Ind(A)$ denotes the indicator of the set $A.$

\subsection{Terminology} All results  are stated in the terms of homogeneous random fields (of course, they are valid for stationary random sequences and processes, too).  If $m=1$, the words "homogeneous random field" mean "stationary random process" or "stationary random sequence"; the words  "ball",   "cube" and "parallelepiped" \\ mean "interval" and,  if $m=2$,  these words mean "sphere" "square" and "paralle\-logram,
 respectively.

\section{Randomization using  uniform distributions on\\ subsets $T_n^l$ of $T$}\label{randomization}
Let $X(t)=(X^l(t),...,.X^d(t)),\;t\in T,$ be a  homogeneous random field on $T$, $\{T_n^l\}$ be    sequences of bounded Borel sets of positive measure in $T\\ (l=1,...,d),\; d\in\N$ (these sequences   may coincide for some or even all $l$).
It is supposed  that $\lambda(T^l_n)\to \infty$ as $n\to \infty.$

 We consider the $T_n$-valued random vectors  $\tau^l_{n, 1},...,\tau^l_{n,k_n},\;\;l=1,...,d,$ $n\in \N$, introduced in Subsect. \ref{randomizing}.
\begin{lemma}\label{int}
	If $f$ is a measurable function on $\R$ such that $E_{X}[|f(X(0))|]<\infty$, then $E_{X,\tau}[|f (X(\tau^l_{n,i}))|] <\infty$ and with $P_X$-probability 1\\ $E_\tau[|f (X(\tau^l_{n,i}))|] <\infty$.
\end{lemma}
\begin{proof} It follows immediately from
	 the Fubini - Tonelli Theorem.
	 \end{proof}

 Assume that for each $l\in\{1,...,d\}$ and for each natural number $n$ the    random vectors $\tau^l_{n,i},\;i=1....,k_n $, introduced in Subsect. \ref{randomizing}, are  uniformly distributed (with respect to $\la$)  on the  set $T_n^l$.   In this section  we use  the following  estimators of $\mu^l$ and $(\sigma^l)^2$:
 \begin{gather*}\label{est}
 M_n^l:=\frac1{\la(T_n^l)}\int_{T_n^l}X^l(t)\la(dt) ,\\ V_n^l:=\frac1{\la(T_n^l)}\int_{T_n^l}(X^l(t)-M_n^l)^2\la(dt)=\nonumber
 \frac1{\la(T_n^l)}\int_{T_n^l}(X^l(t))^2\la(dt)-(M_n^l)^2.
\end{gather*}
 If $\omega\in\Omega_X$ is fixed, the random variables $X^l( \tau^l_{n,i},\omega)$ form $d$  triangular arrays and for each $n$   they are independent and identically distributed.
For each measurable function $f$ on $\R$,  such that $  f[X^l(t)]$ is integrable on $ T_n^l$,
 \begin{equation*}\label{Ew}
 E_\tau [f(X^l(\tau^l_{n,i}))]= \frac1{\la(T_n^l)}\int_{T_n^l}f[X^l(t)]\la(dt),\;\; i=1,...,k_n,\;\;l=1,...,d.
   \end{equation*}
In particular,
  \begin{equation*}\label{E}
 E_\tau [X^l(\tau^l_{n,i})]=M^l_n;\;\;  Var_\tau [X^l(\tau^l_{n,i})]=V^l_n.
  \end{equation*}
The above relation  implies:
 \begin{gather}\label{Et} E_\tau[\sum_{i=1}^{k_n}f(X^l(\tau^l_{n,i}))]=
k_n\frac1{\la(T_n^l)}\int_{T_n^l}f[X^l(t)]\la(dt) \ (l=1,...,d).
 \end{gather}

\begin{definition}\label{admis}
We say that a sequence of Borel sets $\{T_n\}$ is\emph{ pointwise avera\-ging} if
 the following \emph{Pointwise Ergodic Theorem} (PET) holds with this sequence:

\emph{Let $f(\cdot)$  be a measurable function on $\R$.  If $X$ is a   scalar  homogeneous  random field on $T$ over $ (\Omega_X,\F_X, P_X)$ with    $E_X[ |f(X(0))|]<\infty$ and $\mathcal{I}_X$  is  the $\sigma$-field of shift-invariant $\mod( P_X)$ events in $\F_X$, then with $P_X$-probability 1}
\begin{equation*}\label{PET}
\lim_{n\to\infty} \frac1{\la(T_n)}\int_{T_n}f(X(t))\la(dt)=E_X[f(X(0))|\mathcal{I}_X].
\end{equation*}
%(if $X$ is ergodic, then $E_X[f(X(0))|\mathcal{I}_X]=E_X[f(X(0))]$ with $P_X$-probability1).
\end{definition}

\begin{example}\label{convex}
In $\R^m$ each increasing sequence of  bounded  convex sets \\ $T_n^l\in \R^m$,    containing  balls $B_n$ of radii $r(B_n)\to\infty $,  is pointwise averaging  (see Corollary 3.3 in Ch. 6 in \cite{T} ). In particular, any sequence of concentric balls $B_n$ with  $r(B_n)\to\infty $ and any increasing sequence of parallelepipeds   $T_n\subset \R^m$ with infinitely growing  edges  is pointwise averaging (in particular,  the sequence  of cubes $[0,n]^m$ has this property). The intersections  of the mentioned sets with $\Z^m$ form pointwise averaging sequences  in $\Z^m$ (see Subsect. 5.2 in \cite {TS}).
\end{example}

\begin{example}\label{homot} Let $ A$ be a  compact set  in $\R^m$,
  	with $\la(A)>0$ and star-shaped with respect to  $0$. If $s_n\uparrow \infty$, the sequence of homothetic sets
$\{s_nA\}$ is  pointwise averaging (see Example 2.9 in  Ch. 5  in \cite{T} or Subsection 5.4.1 in \cite {TS}).
\end{example}
\begin{remark}\label{MET}
  Along with the pointwise averaging sequences of sets, \emph{mean  averaging} sequences may be considered, i.e. sequences $T_n$ for which   The \emph{Mean Ergodic Theorem} (MET) for each homogeneous random field with $E[(X(0))^2]<\infty$:
 $$E_X\left[|\frac1{\la(T_n)}\int_{T_n}X(t)\la(dt)-  E_X[X(0)|\emph{I}_X]|^2\right]\to 0. $$
Under our condition $E[|X(0)|^{2}]<\infty,$ each pointwise averaging sequence is mean averaging (see Corollary 2 in Subsect 9.4 in \cite{Lo}). Consider the Hilbert subspace $H_X$ of $L^2(\Omega_X,\F_X,P_X)$, spanned by the random variables $X(t),$\\ $t\in T$; it is invariant with respect to the shift transformations: $W(\gamma_t\omega)\in H_X$ if $W(\omega)\in H_X$; denote by $I$  the subspace of all random variables invariant with respect to all $\gamma_t$. The conditional expectation $E_X[X(0)|\emph{I}_X]=\tilde E_X[X(0)|I]$, the orthogonal projection of $X(0)$ onto $I$. This  random variable equals  $E[X(0)]$, if $X$ is "wide-sense ergodic", i.e., if $I=\R$; of course this condition is much weaker than ergodicity, which assures that   the subspace of all  $\gamma$-invariant random variables in $L^2(\Omega_X,\F_X,P_X)$ coincides with $\R$.   The class of the mean averaging  sequences is rather wide; in particular, the  monotonicity condition in Example \ref{convex} is redundant.  The MET is essentially used in the sequel (see also Remark \ref{k_n} where the speed of convergence in the MET is discussed).
\end{remark}

\section{Randomized Lindeberg condition}\label{Lindeberg}

Let us remind that the main probability space has the following form\\ $(\Omega_{X,\tau},\F_{X,\tau},P_{X,\tau}) =(\Omega_X\times \Omega_\tau,\F_{X}\times \F_\tau,P_{X}\times P_\tau)$.

\begin{proposition}\label{lindeberg}
	Let  $X$ be a homogeneous random field, $E[(X(0))^2]<\infty$,
 and  let $\{T_n\}$ be a pointwise averaging sequence of sets.	
	Then for $P_X$-almost all $\omega \in\Omega_X$   the Lindeberg condition is fulfilled:

for each $\var$ as $n\to \infty$
	\begin{gather*}
		L_n(\omega ): =
		\frac{\sum_{i=1}^{k_n}E_\tau[ [X(\tau_{n,i},\omega)-M_n(\omega)]^2{\mathbf 1}_{B_n}]}
		{{k_n}V_n(\omega)}\to 0,
	\end{gather*}
where $B_n =\{|X(\tau_{n,i},\omega)-M_n(\omega)|>\var k_n^\frac12 (V_n)^{\frac12}\}.$
\end{proposition}
\begin{proof}
For a fixed "good" $\omega$, we will  apply the Lindeberg theorem (see, e.g., Theorem 27.2 in \cite{B}) to the random variables over the space $(\Omega_\tau,\F_\tau, P_\tau)$. By relation \eqref{Et},  for each $\omega\in \Omega_X$ and for each $\var>0$, the Lindeberg  fraction for the random variables $X(\tau_{n,i}(w),\omega),\; i=1,...,k_n,$	
is given by
	\begin{gather*}
	L_n(\omega ) =
	\frac{1}{V_n(\omega)}\frac{1}{\la(T_n)}\int_{T_n}
	(X(t,\omega) -M_n(\omega))^{2}{\mathbf 1}_{C_n(t)}\la(dt),
	\end{gather*}
where $C_n(t) = \{|X(t,\omega) -M_n(\omega)|>\var k_n^\frac12 (V_n)^{\frac12}\}.$	

We have to prove that with $P_X$-probability 1 $L_n(\omega)\to 0$.
Since the sequence $\{T_n\}$ is pointwise averaging, by the PET, with $P_X$-probability 1
\begin{equation}\label{li1}
	M_n(\omega)\to z(\omega),
		\end{equation}
\begin{gather}	\label{li2}
	 V_n (\omega)= \frac1{\la(T_n)}\int_{T_n}(X(t,\omega) -M_n(\omega))^{2}\la(dt)\to v(\omega),
	\end{gather}
where
\begin{gather}
z(\omega) = E_X[X(0,\omega)|\mathcal{I}_X](\omega),
\end{gather}
\begin{gather}
 v(\omega) =  Var_X[(X(0),\omega)^{2}| \mathcal{I}_X](\omega)>0.\label{li3}
\end{gather}
It is clear that
\begin{gather*}
	L_n(\omega )\leq \Sigma^{(1)}_n + \Sigma^{(2)}_n,
\end{gather*}
where
$$
\Sigma^{(1)}_n = \frac{1}{V_n(\omega)}\frac{2}{\la(T_n)}\int_{T_n}
(X(t,\omega)-z(\omega))^{2}{\mathbf 1}_{C_n(t)}\la(dt),
$$
$$
\Sigma^{(2)}_n = \frac{1}{V_n(\omega)}\frac{2}{\la(T_n)}\int_{T_n}
(z(\omega)-M_n)^2{\mathbf 1}_{C_n(t)}\la(dt).
$$
Due to (\ref{li1})    with $P_X$-probability 1
\begin{equation}\label{sigma2}
\Sigma^{(2)}_n \leq (z(\omega)-M_n(\omega))^2\frac2 {V_n(\omega)}\to 0.
\end{equation}
Consider $\Sigma^{(1)}_n.$ Denote by $\Lambda\in \F_X$ the set of $P_X$-measure 1 in  $\Omega$ where the   limits \eqref{li1} and \eqref{li2} exist, and fix some  $\omega\in\Lambda$. It is clear that $k_nV_n(\omega)\to \infty$ as $n\to\infty$.
For each $C>0$, let us also fix some $m_C(\omega)\in\N$   such that,  if $n>m_C(\omega)$, then $|M_n(\omega)-z(\omega)|<\frac12\var C $ and $(k_nV_n(\omega))^\frac12>C>0$.

As for $n>m_C(\omega) $
$$
\{|X(t,\omega)-M_n|>\var k_n^\frac12 (V_n)^{\frac12}\}\subset
\{|X(t,\omega)-z(\omega)|>\frac{\var C}{2}\}\cup \{|z(\omega)-M_n|>\frac{\var C}{2}\},
$$
we have:
\begin{equation}\label{sigma1}
\Sigma^{(1)}_n \leq \delta^{(1)}_n + \delta^{(2)}_n,
\end{equation}
where
$$
\delta^{(1)}_n = \frac{1}{V_n(\omega)}\frac{2}{\la(T_n)}\int_{T_n}
(X(t,\omega)-z(\omega))^{2}
{\mathbf 1}_{\{|X(t,\omega)-z(\omega)|>\frac{\var C}{2}\}}\la(dt),
$$
$$
\delta^{(2)}_n = \frac{1}{V_n(\omega)}\frac{2}{\la(T_n)}\int_{T_n}
(X(t,\omega)-z(\omega))^{2}
{\mathbf 1}_{\{|z(\omega)-M_n|>\frac{\var C}{2}\}}\la(dt).
$$
By the PET, with $P_X$-probability 1
\begin{equation}\label{delta1}
\lim_n \delta^{(1)}_n = \frac{2}{v(\omega)}
E\left[(X(0,\omega)-z(\omega))^{2}
{\mathbf 1}_{\{|X(0,\omega)-z(\omega)|>\frac{\var C}{2}\}}\right],
\end{equation}
and due to (\ref{li1})
\begin{equation}\label{delta2}
\lim_n \delta^{(1)}_n = \frac{2}{v(\omega)}\lim_n	\frac{1}{\la(T_n)}\int_{T_n}
	(X(t,\omega)-z(\omega))^{2}\la(dt)\cdot \lim_n
{\mathbf 1}_{\{|z(\omega)-M_n|>\frac{\var C}{2}\}}	= 0.
\end{equation}
It follows from (\ref{sigma2})-(\ref{delta2}) that
$$
\limsup_n 	L_n(\omega )\leq 2E	\left[(X(0,\omega)-z(\omega))^{2}
{\mathbf 1}_{\{|X(0,\omega)-z(\omega)|>\frac{\var C}{2}\}}\right].
$$
Letting $C$ tend to infinity, we finally obtain with $P_X$-probability 1
$$
\lim_n 	L_n(\omega ) = 0.
$$
\end{proof}

\section{Generalizations of the Central Limit Theorem}\label{generalCLT}
\subsection{On the CLT} \label{intr_CLT}The Central Limit Theorem (CLT) is one of the remarkable statements of  Probability. Originally proved for sequences of independent random variables, in many works this theorem was generalized to stationary random processes and homogeneous random fields  under some additional assumptions: Markov,  strong mixing and/or rather  strong moment conditions, etc.  (see, e.g., \cite{Bo,Br1,Br2,Bu, BZ,BD, CDV, De, DL, I,I2, IvLe, Le, Ma, MVW, MPU,MP, PZ}, \cite{To}-\cite{ZRP}).

 However, the CLT may  fail even   when the sequence $X_1,X_2,...,$ is  stationary, quite  strongly   mixing, orthogonal  and  with
 $\lim\frac 1n\text{Var}(\sum_{i=1}^n X_i)=\sigma^2>0$ \cite{Herrndorf}; two other interesting counterexamples are provided in \cite{Br2}. Earlier  Chung \cite{Chung2} and Davydov \cite{Davydov1,Davydov2} proved that the CLT may fail for strictly stationary irreducible aperiodic Markov chains with a countable set of states (these and other counterexamples can be found in vol. 3, Chapters 30 and 31 in the book  \cite{Br2}). Ibragimov and Linnik \cite{IL} have  constructed a strongly mixing  stationary sequence  of random variables $X_1,X_2,...,$  with finite variances such that the self-normalized sums $(Var[\sum_{i=1}^nX_i])^{-\frac12}\sum_{i=1}^nX_i$ converge in distribution to a non-normal random variable (see Ch. 19, \S 5 therein). A rather common case when the CLT  fails is when
$\lim\frac 1n\text{Var}[ \sum\limits_{i=1}^{n}{X_i}] =0$ (see, e.g., \cite{Der}, p.153).

 A well-known  source   of counterexamples for the classical CLT for strict sense stationary random sequences is the  class of  coboundaries  with respect to measure preserving transformations. Let $\gamma: \omega \mapsto \gamma \omega$ be an ergodic measure preserving transformation of a probability space $(\Omega,\F,P)$; let $\al\ge1$; then for each $f\in L^\al(\Omega,\F,P), f\ne 0 $ a.s.,  the random sequence $X_i^{(f)}=f(\gamma^i\omega), i=1,2,..., $ is an ergodic strict sense stationary sequence and $E [|X_1^{(f)}|^\al]<\infty$  (see,  e.g., \cite{Kr}, \S 1.4). Let $g\in L^\al(\Omega,\F,P)$
 and  let $f$ be the \emph{coboundary} of $g$: $f(\omega)=g(\gamma\omega)-g(\omega)$.
It is clear that  $E(X_i^{(f)})=0,$ $i=1,2,...$. We have also: $\sum_{i=1}^nX_i^{(f)}=\sum_{i=1}^nf(\gamma^i\omega)=g(\gamma^{n+1}\omega)-g(\gamma\omega)$, hence $E[|\sum_{i=1}^nX_i^{(f)}|^\al\le 2^\al||g||_{L^\al}^\al]$,  and $E[|n^{-\beta}\sum_{i=1}^nX_i^{(f)})|^\al]\to 0$  if $\beta>0$ (when $\al>1$, the converse is true: if $E[f]=0$ and\\ $\sup_nE[(\sum_{i=1}^nX_i^{(f)})^\al]<\infty$, it is a  coboundary of some $g\in L^\al$; see Lemma 5 in \cite{Browder}); the set of coboundaries   is dense in the space $L^\al(\Omega,\F,P)\ominus\mathbb R$, by the ergodic decomposition. It is evident that, if $f$ is a coboundary of a function  $g\in L^\al$, then  $n^{-\beta}\sum_{i=1}^nX_i^{(f)}\to 0$ in probability for any $\beta>0$, and,  if  $g$ is bounded, then this is true in the sense of a.s. convergence.

There is an  extensive literature related to limit theorems  for ``long memory" (or ``long-range dependent") stationary sequences, where the slowly decreasing  dependence between receding terms is characterized by  the rate of  decrease of the correlation function or by properties  of the spectrum near 0) - see \cite{BFGK},  \cite{G}, \cite{S} and the references therein.

\subsection  {Main results of this section}   In the previous  work by the second author \cite{T1},
CLTs for  homogeneous random fields  on $\R^m$ and $\Z^m \ (m\ge 1)$,   satisfying the condition: $E[|X(0)|^{2+\delta}]<\infty$ have been proved. In the present  article the authors prove that these theorems are valid under the weaker  condition $E[(X(0))^{2}]<\infty$.

In this paper  we present   randomized versions  of the   CLT, which are valid  for  each ergodic homogeneous measurable random field $X(\cdot)$   on $\R^m$ or\\ $\Z^m \ (m\ge 1)$, (in particular, for each  ergodic stationary random process and each ergodic stationary random sequence) with a finite second moment (in some versions ergodicity may be omitted). Specifically, observations of the random field at randomly chosen points are used.

If the field X is multivariate and each of  its components satisfies the above conditions, another feature is obtained by the randomization: in Theorems \ref{CLT1}-\ref{CLT} and Corollaries 1, 2 the components of the limit normal vector are independent whatever are the components of $X$.
\emph{When $E[|X(0)|^2]<\infty$, these theorems are valid also in all cases, mentioned in Subsection 1.1, when the conventional CLT fails.}

For example, consider the  stationary random sequences  $\{X_i^{(f)}\}$, which have been introduced in Subsection \ref{intr_CLT} (each sequence is specified by  some coboundary  $f\in L^{2}(\Omega,\F,P))$; we have: $E[|X_1^{(f)}|^{2}]<\infty$; as mentioned above, the set  of  coboundaries is dense in $L^{2}(\Omega,\F,P)\ominus\mathbb R$, and the conventional CLT fails for  $\{X_i^{(f)}\}$, if $f$ a coboundary: the limit is degenerate ($\delta_0$) and not normal. However, the randomized CLTs are valid for all sequences $\{X_i^{(f)}\}$ with $E[|X_i^{(f)}|^2<\infty$.

Randomization allowed us
to assume that the field is only  homogeneous   (in some statements it is assumed   that the field is also ergodic) and its second moment is finite.

The main tools, used in the proofs, are the Lindeberg CLT and the Pointwise and Mean Ergodic Theorems.

To illustrate our results, we state a simple corollary of Theorem 3.
\vskip .25cm
\emph{Let $X(t), t \in  \R$, be an ergodic  stationary measurable random process and $E[(X(0))^{2}]
 <\infty$; denote
$\sigma^2 = V ar[X(0)], M_n =
\frac1n
\int_0^n X(t)dt$. Let $\tau_{n,i}\\ (n = 1, 2, ...$, and $i = 1, ..., n)$
be random variables, independent of $X$, and, for each $n$, independent of each
other and uniformly distributed on $[0, n]$. Then, if $n \to\infty$,
\begin{equation*}
\frac{\sum_{i=1}^{n}(X(\tau_{n,i})-M_n)}{\sigma\sqrt n}\Longrightarrow N(0,1).
%\text{in the joint $(X,\tau)$-distribution.}
\end{equation*}}

A slightly more complex corollary with $\mu = E[X(0)]$ instead of $M_n$ can be
derived from Theorem 4; as usual, if $\sigma$ is known, this statement may be used
for consistent statistical inference on the expectation $\mu$.
\subsection{Randomized Central Limit Theorem}\label{generCLT}

\begin{assumption}\label{assump}
In Theorems \ref{CLT1}--\ref{CLT5},    all $X^l,\;l=1,...,d,$  are  homogeneous random fields  and, for all $l$,
\begin{equation}\label{mom2}
E_X[|X^l(0)|^{2}]<\infty.
\end{equation}
If $X^l$ is not ergodic,  we assume $Var_X[X^l(0)|\mathcal{I}^l_X]>0 $ with $ P^l_{X}$-probability 1, $l=1,...,d$.

We remind that $k_n\in\N,\;  k_n\uparrow\infty$.

 For each $l$,\  $\{T_n^l\}$ is a pointwise averaging sequence of sets.
\end{assumption}
\begin{remark}
 We will prove that, together with the conditions mentioned above, in the non-ergodic case, the condition  $Var_X[X^l(0)|\mathcal{I}^l_X]>0   \ P^l_X{\rm -a.s.},\\  l=1,...,d, $ is \emph{sufficient} for the relation \eqref{Ly5} to hold with $P_X$-probability 1. And, for each fixed  $l$,  it  is also \emph{necessary} for this statement to hold for the component $X^l$. Indeed, if the variance of the conditional distribution $Var_X[X^l(0)|\mathcal{I}^l_X]=0$, then $X^l(0)=E_X[ X^l(0)|\mathcal{I}^l_X]$ \ $P_X^l$-a.s., hence $X^l(0)$ is $\mathcal{I}^I_X$-measurable  and $X^l(t,\omega^l)=X^l(0,\gamma_t^l\omega^l )=X^l(0,\omega^l )$ \  $P_X^l$-a.s. so, for each $t\in T$, $X^l(t)=X^l(0)$ \  $P_X^l$-a.s. It is clear, that this dependence is too strong for any version of the   CLT   to hold for $X^l$. In this case, it becomes senseless: note that  $M^l_n=X^l(0), V_n^l=0;$  therefore,   the  expression
  $$\frac{\sum_{i=1}^{k_n}(X^l(\tau^l_{n,i})-M_n^l)}{k_n^\frac12(V^l_n)^{\frac12}}=\frac{{k_n}(X^l(0)-M^l_n)}{k_n^\frac12(V^l_n)^{\frac12}}$$
 (see \eqref{Ly5}  ) is of type $\frac00$; conditions  \eqref{cond16} and \eqref{cond} for $X^l$ do not hold, and the expressions \eqref{Ly7}, \eqref{Ly3} for $X^l$ are of type $\frac10$.
\end{remark}
Let us remind that the probability space
$$
(\Omega_{X,\tau},\F_{X,\tau},P_{X,\tau}) =(\Omega_X\times \Omega_\tau,\F_{X}\times \F_\tau,P_{X}\times P_\tau).
$$

\begin{lemma}\label{p}
1. Let  $A_n(z)$ be a sequence of events in $\Omega_{X,\tau}$ depending on $z\in\R^d$; for each $\omega\in \Omega_X$ denote: $A_n(\omega,z)=\{t\in\Omega_\tau: (\omega, t)\in A_n(z)\}$, the $\omega$-section of $A(z)$; if  for each $z\in\R^d$    the limit
$$\lim_{n\to\infty} P_\tau (A_n(\omega,z))=\psi(z)$$
exists for  $P_X$-almost all $\omega$,  then for each $z\in\R^d$
\begin{equation}\label{conv}
 \lim_{n\to\infty}P_{X,\tau}(A_n(z))=\psi(z).
\end{equation}
2.  Let  $\{f_n(X,\tau)=(f_n^1(X,\tau),...,f_n^d(X,\tau)), n\in N\}$ be a sequence of  random vectors over $\Omega_{X,\tau}$ and let $Z=(Z^1,...,Z^d)$  be a random vector  over the same probability space. If   $f_n(X,\tau)\to Z$  in $P_\tau$-distribution  with $P_X$-probability 1, then  $f_n(X,\tau)\to Z$ in $P_{X,\tau}$-distribution.
\end{lemma}
\begin{proof}
1. By   the Fubini theorem, $A_n(\omega,z)\subset \F_\tau$ for $P_X$-almost all $\omega\in \Omega_X$. By virtue of the Fubini theorem and  the Lebesgue Dominated Convergence theorem,  we have for each $z\in\R^d$:
\begin{gather*}\label{Ly2}\lim_{n\to\infty}P_{X,\tau}(A_n(z))=\lim_{n\to\infty}E_{X,\tau}[1_{A_n(z)}]=\lim_{n\to\infty}E_{X}[E_{\tau}[1_{A_n(z)}]]=\\
\lim_{n\to\infty}E_{X}[P_{\tau}(A_n(\omega,z)]=
E_{X}[\lim_{n\to\infty}P_{\tau}(A_n(\omega,z)]=E_{X}\psi(z)=\psi(z).
\end{gather*}

2. Apply Statement  1 with
\begin{gather*}
  A_n(z):=\{(\omega,t)\in\Omega_{X,\tau}: f_n^1(\omega,t)\le z^1,...,f_n^d(\omega,t)\le z^d\}, \\ \psi(z)=P_X(Z^1\le z^1,...,Z^d\le z^d).
\end{gather*}
\end{proof}
We denote: $Z =(Z^1,...,Z^d)$ where $Z^1,...,Z^d$ are independent standard normal random variables.
\begin{theorem} \label{CLT1} Under Assumptions  1,
 with $P_X$-probability 1
\begin{equation}\label{Ly5}\left(\frac{\sum_{i=1}^{k_n}(X^1(\tau^1_{n,i})-M_n^1)}{k_n^\frac12(V^1_n)^{\frac12}},...,\frac{\sum_{i=1}^{k_n}(X^d(\tau^d_{n,i})-M_n^d)}
{k_n^\frac12(V^d_n)^{\frac12}} \right)\overset {P_{\tau}}\Longrightarrow Z;%\sim N(0,I)
\end{equation}
\noindent if all $X^l$ are ergodic, then with $P_X$-probability 1

\begin{equation}\label{Ly11}\left(\frac{\sum_{i=1}^{k_n}(X^1(\tau^1_{n,i})-M_n^1)}{k_n^\frac12\sigma^1}, ...,\frac{\sum_{i=1}^{k_n}(X^d(\tau^d_{n,i})-M_n^d)}{k_n^\frac12\sigma^d}\right)\overset {P_{\tau}}\Longrightarrow Z.  
\end{equation}
\end{theorem}

\begin{proof}
We may apply  the Lindeberg  theorem and obtain:    with $P_X$-probability 1
for each $l$
\begin{gather}\label{Lyap1}
	\frac{\sum_{i=1}^{k_n}(X^l(\tau^l_{n,i})-M^l_n)}{k_n^\frac12(V^l_n)^{\frac12}}\overset {P_{\tau}}\Longrightarrow  Z^l.
\end{gather}
If $X$ is ergodic, then  $ (V^l_n)^\frac12\to \sigma^l $ with $P_X$-probability 1, and, by the Slutsky theorem (\cite{L}, Theorem 23.3),
{with $P_X$-probability 1}
\begin{gather}\label{Lyap2}
	\frac{\sum_{i=1}^{k_n}(X^l(\tau_{n,i})-M^l_n)}{k_n^\frac12\sigma^l}
\overset {P_{\tau}}\Longrightarrow  Z^l.	
\end{gather}
Since for each $n$ the $d$ vectors  $(\tau^l_{n,i},i=1,...,k_n),\;\;l =1,...,d,$ are mutually independent, using  \eqref{Lyap1} and \eqref{Lyap2} we obtain
the convergence (\ref{Ly5}), respectively  (\ref{Ly11}), of vectors.
\end{proof}

Since the limiting distributions in (\ref{Ly5}) and  (\ref{Ly11})
 with $P_X$-probability 1 are the same, we immediately deduce from the
 previous theorem the non-conditional weak convergence:

\begin{theorem}\label{CLT2} Under Assumptions 1,
	
	\begin{equation*}\left(\frac{\sum_{i=1}^{k_n}(X^1(\tau^1_{n,i})-M_n^1)}{k_n^\frac12(V^1_n)^{\frac12}},...,\frac{\sum_{i=1}^{k_n}(X^d(\tau^d_{n,i})-M_n^d)}
		{k_n^\frac12(V^d_n)^{\frac12}} \right)\overset {P_{X,\tau}}\Longrightarrow  Z;%\sim N(0,I)
	\end{equation*}
	
	\noindent if all $X^l$ are ergodic, then with $P_X$-probability 1
	\begin{equation*}\label{Ly9}\left(\frac{\sum_{i=1}^{k_n}(X^1(\tau^1_{n,i})-M_n^1}{k_n^\frac12\sigma^1}, ...,\frac{\sum_{i=1}^{k_n}(X^d(\tau^d_{n,i})-M_n^d)}{k_n^\frac12\sigma^d}\right)\overset {P_{X,\tau}}\Longrightarrow   Z.
	\end{equation*}
\end{theorem}

The following Theorems \ref{CLT3}  and \ref{CLT4} can be deduced subsequently from Theorems \ref{CLT1}-\ref{CLT2} literarily as Th. 3 and Th. 4 in \cite{T1}.
\begin{theorem}\label{CLT3} In addition to the above Assumptions \ref{assump}, assume that

\begin{equation}\label{cond16}
	k_n^\frac12(M^l_n-\mu^l)  \to 0 \ \text{with $P_X$-probability 1},\ l=1,...,d  \footnote{See \S\ref{remarks} for remarks on this condition and condition  \eqref{cond17}.}^).
\end{equation}
Then with $P_X$-probability 1

\begin{equation*}\label{Ly7}\left(\frac{\sum_{i=1}^{k_n}(X^1(\tau^1_{n,i})-\mu^1)}{k_n^\frac12(V^1_n)^{\frac12}},...,\frac{\sum_{i=1}^{k_n}(X^d(\tau^d_{n,i})-\mu^d)}
	{k_n^\frac12(V^d_n)^{\frac12}} \right)\overset {P_{\tau}}\Longrightarrow  Z;%\sim N(0,I)
\end{equation*}

\noindent if all $X^l$ are ergodic, then with $P_X$-probability 1
\begin{equation*}\label{Ly8}\left(\frac{\sum_{i=1}^{k_n}(X^1(\tau^1_{n,i})-\mu^1)}{k_n^\frac12\sigma^1}, ...,\frac{\sum_{i=1}^{k_n}(X^d(\tau^d_{n,i})-\mu^d)}{k_n^\frac12\sigma^d}\right)\overset {P_{\tau}}\Longrightarrow  Z.
\end{equation*}
\end{theorem}

\begin{theorem}\label{CLT4} Let the conditions stated in Assumptions \ref{assump} be fulfilled
	and
	\begin{equation}\label{cond17}
		k_n^\frac12(M^l_n-\mu^l) \overset {P_X} \Longrightarrow 0.
	\end{equation}
	Then
	\begin{equation*}\label{Ly3}\left(\frac{\sum_{i=1}^{k_n}(X^1(\tau^1_{n,i})-\mu^1)}{k_n^\frac12(V^1_n)^{\frac12}},...,\frac{\sum_{i=1}^{k_n}(X^d(\tau^d_{n,i})-\mu^d)}
		{k_n^\frac12(V^d_n)^{\frac12}} \right)\overset {P_{X,\tau}}\Longrightarrow Z;%\sim N(0,I)
	\end{equation*}
	
	\noindent if all $X^l$ are ergodic, then
	\begin{equation*}\label{Ly4}\left(\frac{\sum_{i=1}^{k_n}(X^1(\tau^1_{n,i})-\mu^1)}{k_n^\frac12\sigma^1}, ...,\frac{\sum_{i=1}^{k_n}(X^d(\tau^d_{n,i})-\mu^d)}{k_n^\frac12\sigma^d}\right)\overset {P_{X,\tau}}\Longrightarrow Z.
	\end{equation*}
\end{theorem}

Condition \eqref{cond17} puts some restrictions to   the growth of the number $k_n$ of observations of the values of the field $X$. It is often of the form $k_n=o(\la(T_n))$ (see Remark \ref{k_n}). On the other hand, the accuracy of the approximation of the normal distribution by the CLT is higher when $k_n$ is large. An obvious way to increase $k_n$ without violation of this condition   is to increase $\la(T_n)$ proportionally. Another way to increase the number of observations without breaking this condition puts more work for the statistician;
 it is based on a special case of Theorem \ref{CLT4}, when, instead of $X$, an auxiliary $\R^{wd}$-valued random field $Y $ is considered, in which  each component $X^l$ participates $w$ times.

\begin{corollary}\label{d_gen}
Let the conditions of Theorem \ref{CLT4} be fulfilled. Consider the  family of independent random variables\\ $\{\tau^{l,u}_{n,i}, n\in\N,\;\;1\le i\le k_n,1\le l\le d,\; 1\le u\le w\}$ which, for each  $n,l,u,i,$ are distributed uniformly on  $T_n^l$.
Then
\begin{equation*}\label{Ly6}\frac{\sum_{i=1}^{k_n}(X^l(\tau^{l,u}_{n,i})-\mu^l)}{k_n^\frac12 (V_n^l)^\frac12} \overset {P_{X,\tau}}\Longrightarrow Z^{l,u}\ ,\;\;l=1,...d,\;\;u=1,...w.
\end{equation*}
where all random variables $Z^{l,u},\;\;l=1,...,d,\;\;u=1,...,w,$ are standard normal and independent.

If the field $X$ is ergodic, the "empirical" standard deviations $V_n^l$ can be replaced by the  "theoretical" standard deviations $\sigma^l$.
\end{corollary}
\begin{proof}
  Apply Theorem \ref{CLT4} to the random fields $Y^{l,u}(t):=X^l(t),\;\;t\in T,\;\; l=1,...,d,\;\;u=1,...,w.$
\end{proof}
 We consider  a generalization of this corollary to various parameters of the marginal or multidimensional distributions of the field $X$; for simplicity we assume that $X$ is $\R$-valued. Let $\theta^l,l=1.,,,,d,$ be   parameters of the distribution of a vector $(X(t_1),...,X(t_k))$. We assume that there exists a $\B(\R^k)$-measurable  functions $f^l(x_1,...,x_k)$ such that
 \begin{equation*}
  E_X[|f^l(X(t_1),...,X(t_k))|^2]<\infty,
\end{equation*}
and $\theta^l=E_X[f^l(X(t_1),...,X(t_k))].$ Denote:
 $$ \sigma_f^l:=[Var_X[f^l(X(t_1),...,X(t_k))]^\frac12 , $$ $$M^l_{f,n}:=\frac1{\la(T_n)}\int_{T_n}f^l(X(t_1+t),...,X(t_k+t))\la(dt), $$
 $$V^l_{f,n}:=\frac1{\la(T_n)}\int_{T_n}(f^l(X(t_1+t),...,X(t_k+t))-M^l_{f,n})^2\la(dt).$$
For each $l$ we consider the  random field
 $$Y_f^l(t):= f^l(X(t_1+t),...,X(t_k+t)),\;t\in T,$$
  over  $(\Omega_X,\F_X,P_X)$). Note that $Y_f^l(0):= f^l(X(t_1),...,X(t_k+t))$, so \emph{the condition
 $E_X|Y_f^l(0) |^2<\infty $ is  fulfilled}; the finite-dimensional distributions of the field $Y_f^l$ coincide with   finite-dimensional distributions of $X^l$, and   $Y_f^l$ is generated by  the  shift transformations: by Property \eqref{covariant},
  \begin{gather*}
   Y_f^l(0,\gamma_t\omega)= f^l(X(t_1,\gamma_t\omega),...,X(t_k,\gamma_t\omega)) =\\
   f^l(X(t_1+t),\omega),...,X(t_k+t),\omega))=Y_f^l(t,\omega);
  \end{gather*}
  therefore, the field $Y_f^l$ is  homogeneous, and,  if the field $X^l$ is ergodic, the field  $Y_f^l$, is ergodic, too.

For example, if $\theta$ is a   mixed moment, i.e.
$$\theta= E[(X (t_1))^{v_1}...(X (t_k))^{v_k}],\; (v_1,...,v_k\in\N),$$
we have: $ f^l(x_1,...,x_k)=(x (t_1))^{v_1}...(x (t_k))^{v_k}$,
  $$Y_f(t)=(X (t_1+t))^{v_1}...(X (t_k+t))^{v_k}.$$

  Application of Corollary \ref{d_gen} to the    fields  $Y_f^l(t)$ brings us to the following statement.
\begin{corollary} Let $X$ be a     homogeneous scalar random field.
 Under the conditions  of Corollary \ref{d_gen} (with $Y^l_{f}, M^l_{f,n},V^l_{f,n},\theta^l$  instead of $ X^l,M^l_n,V^l_{n},\mu^l$,\linebreak
 respectively), for $l=1,...,d,\;\;u=1,...,w,$
\begin{equation*}\frac{\sum_{i=1}^{k_n}( f^l(X(t_1+\tau^{l,u}_{n,i}),...,X(t_k+\tau^{l,u}_{n,i}))-\theta^l) }{k_n^\frac12(V^l_{f,n})^\frac12} \overset {P_{X,\tau}}\Longrightarrow Z^{l,u}.
\end{equation*}
where all random variables $Z^{l,u}$ are standard normal and independent.

If $X$ is ergodic, the "empirical" standard deviations  $(V^l_{f,n})^\frac12$ can be replaced by the "theoretical" ones $ \sigma^l_f$.
\end{corollary}

Now let  $\theta:=  F(t_1,...,t_k; x^1,...,x^k):=P(X(t_1)\le x^1,...,X(t_k)\le x^k) $, the $k$-dimensional distribution function of the random field $X$ with  fixed  "time" values $t_1,...,t_k\in \R^m$ and fixed space values $x^1,...,x^k\in\R$. Let %$ f(x_1,...,x_k)=Ind\{ y_1,...,y_k:y_1\le x_1..., y_k \le x^k\}$,
	$Y(t)=Ind\{X (t_1+t)\le x^1..., X (t_k+t)\le x^k\}$.
  Note that
  $$
  E_X[Y(0)]= F(t_1,...,t_k; x^1,...,x^l).
  $$

   Denote:
   $$
   M_{n}:=
  \frac1{\la(T_n)}\int_{T_n}Y(t)\la (dt),\;\;\;\;
    V_{n}:=\linebreak\frac1{\la(T_n)}\int_{T_n}(Y(t)-M_{n})^2\la(dt),
    $$
    $ \textbf{t}:=\{t_1,...,t_k\}, x:=(x_1,...,x_k)$.

 \begin{corollary} Let $X$ be an ergodic     homogeneous scalar random field. Let $\sqrt{k_n}(M_{n}- F(\textbf{t}; x)))\to 0$ in $P_X$-probability.  Consider the  family of independent random variables $\{\tau^{u}_{n,i}, n\in\N,\;\;1\le i\le k_n, 1\le u\le w\}$ which, for each  $n,u,i,$ are distributed uniformly on  $T_n$. Then
 \begin{equation*}\frac{\sum_{i=1}^{k_n}(Ind\{X(t_1+\tau^{u}_{n,i})\le x^1,...,X(t_k+\tau^{u}_{n,i})\le x^k\}- F(\textbf{t}; x))}{k_n^\frac12 V_n ^\frac12} \overset {P_{X,\tau}}\Longrightarrow Z^{u},
	\end{equation*}
	where all random variables $Z^{u},u=1,...,w,$ are standard normal and independent.

The sequence $V_n$ may be replaced by its limit  $F(\textbf{t};x)-(F(\textbf{t};x))^2$.

\end{corollary}

Now we present a randomized version of the CLT when the limit normal distribution is not standard and its covariance matrix coincides with the marginal
covariance matrix of the field X. This statement  is a generalization of   the classical CLT   for i.i.d. random vectors (see Theorem 29.5 in \cite{B}) and can be readily deduced from our Theorem \ref{CLT2} and
\ref{CLT3} (with $d=1$) by using  the Cram$\acute{e}$r - Wold theorem (compare with the proof of the mentioned Theorem 29.5).
\begin{theorem}\label{CLT5} Let the conditions
	of Theorem \ref{CLT1} be fulfilled and, moreover,  let the field $X$ be strict sense stationary and
	ergodic. Let the pointwise averaging sequence $\{T_n\}$ and the randomizing sequence
	$(\tau_{ n,i})$ be the same for all com\-po\-nents $X^l$. Denote by $\Sigma$  the (nonsingular) covariance matrix of the vector $X(0)$
	and let $V$ be a Gaussian vector with mean $0$ and covariance matrix $\Sigma$.

	Then
	\begin{equation*}
		\left( k_n^{-\frac12}\sum_{i=1}^{k_n}(X^1
		(\tau_{ n,i}) - M_n^1),..., k_n^{-\frac12}\sum_{i=1}^{k_n}(X^d
		(\tau_{ n,i}) - M_N^d)\right)\overset {P_{X,\tau}}\Longrightarrow V.
	\end{equation*}
If the condition (\ref{cond17})
 takes place, then also
	\begin{equation*}
	\left( k_n^{-\frac12}\sum_{i=1}^{k_n}(X^1
	(\tau_{ n,i}) - \mu^1),..., k_n^{-\frac12}\sum_{i=1}^{k_n}(X^d
	(\tau_{ n,i}) - \mu^d)\right)\overset {P_{X,\tau}}\Longrightarrow V.
\end{equation*}
\end{theorem}

Our theorems and Examples 1 and 2 imply the following statement.
\begin{corollary} 1. Let $T=\R^m$. Then Theorems 1-3  hold   if $\{T_n^l\}$ is
	
	- an increasing  sequence of bounded convex sets, containing  balls of radii $r_n\to\infty$, or
	
	- a sequence of homothetic sets $\{s_n A\}$ considered in Example \ref{homot}.

	2. The same  is true if $T=\Z^m$ and $T_n^l$ are restrictions onto $\Z^m$ of convex sets in $\R^m$    considered above.
\end{corollary}

\subsection{Rate of convergence}
We remind that $k_n\in\N,\;  k_n\uparrow\infty,\;  \{T_n\}$ is a pointwise averaging sequence of sets and
$$
V_n:=\frac1{\la(T_n)}\int_{T_n}(X(t)-M_n)^2\la(dt).
$$

Let $\mathbf{F}_n$ be the distribution function of the normalaized sum
$$
Z_n = \frac{\sum_{i=1}^{k_n}(X(\tau_{n,i})-\mu)}{k_n^\frac12(V_n)^{\frac12}}
$$
and $\Phi$ be the distribution function of the standard Gaussian law.
We are interested in the estimation of
$$
\Delta_n := \sup_{x\in \R}\Delta_n(x),
$$
where
$$
\Delta_n(x) = |\mathbf{F}_n(x) - \Phi(x)|.
$$
\begin{theorem}\label{rate0}
	Suppose that $\{X(t),\, t\in T\},\;X(t)\in \R,$ is a stationary ergodic field such that
	for some $\delta,\;0<\delta\leq 1,\;  \E|X(t)|^{2+\delta}< \infty.$
	
	There exist a constant $C$ (depending only of $\delta$) such that for all $\var >0$ and every
	$n$
	\begin{equation}\label{rate1}
		\Delta_n \leq P\{V_n < \var \}+ C
		k_n^{-\frac{\delta}{2}} \var^{-\frac{2+\delta}{2}}\E|X(0)|^{2+\delta}.
	\end{equation}	
\end{theorem}
\begin{remark}
Below, in Lemma \ref{L1}, we present several sufficient conditions for estimating the probability $P\{V_n < \var\}.$ Their proof is easily deduced from
Theorem   4.4 in \cite{Tem}.
	\begin{lemma}\label{L1}Let $\{T_n\}$ be a sequence of Borel sets in $T$ with the following property: each $T_n$ is contained in a ball $B(b_n,Cn)$ and $\la(T_n)\ge cn^m\  (b_n\in T, 0<C<\infty, 0<c<C)$. Let $\var<\sigma^2$. We assume that $\mu_4=E[X^4]<\infty$. Suppose the covariance functions:
		$$
		R_{X}(t):=E[ X(t)X(0)] -\mu^2\;\;\; {\mathrm and}\;\;\;R_{X^2}(t):=E[(X(t)X(0))^2]-\mu_2^2
		$$
satisfy the conditions: for some $\beta>0$		
		 $$  R_X(t)=O(|t|^{- \beta }), \;\;\;\; R_{X^2}(t)=O(|t|^{-\frac\beta2}), \;\;\;
		  t\to\infty. $$
		
		Case 1. If $\beta\le m,$
		then  $ P(V_n<\var)\le(\sigma^2-\var)^{-2}O(n^{-\frac \beta2})$.
		
		Case 2. If $\beta=m,$ then  $ P(V_n<\var)\le(\sigma^2-\var)^{-2}O(n^{-\frac m2}(\log n)^\frac 12))$.
		
		Case 3. If $ \beta>m,$ then  $ P(V_n<\var)\le(\sigma^2-\var)^{-2} O(n^{-\frac m2})$.
	\end{lemma}
If we take  $\varepsilon = \sigma^2/2,$ we get
from (\ref{rate1}) the following corollary.
\begin{corollary} Under the conditions of Theorem \ref{rate0} and Lemma \ref{L1} we have:
	{\;}\\
		In Case 1:
$
\Delta_n = O(\max\{n^{-\frac \beta2}, k_n^{-\frac \delta2}\}).
$	

\noindent
	In Case 2:
$
\Delta_n = O(\max\{n^{-\frac m2}(\log n)^\frac 12,
	k_n^{-\frac \delta2}\}).
$	

\noindent
	In Case 3:
$
\Delta_n = O(\max\{n^{-\frac m2}, k_n^{-\frac \delta2}\}).
$	

\end{corollary}
\end{remark}
Now let's move on to the proof
 of the theorem \ref{rate0}.
\begin{proof}
	Since for fixed path of the process
	$X$   the random variables
	$\{X(\tau_{n,i}) \},$\\
	$i=1,\ldots,n,$	 are independent and identically distributed, then denoting
	by  $\mathbf{F}_n^\tau$    the function of conditional distribution
	$Z_n,$  we can apply the estimation
	(\cite{P}, Ch.5, 3, Th. 6 ):
	
	There exists an absolute constant  $C$   such that
	$$
	\Delta_n^\tau(x): = |\mathbf{F}_n^\tau(x) - \Phi(x)|\leq
	\frac{C}{V_n^{(2+\delta)/2} k_n^{\delta/2}}
	E_\tau\{|X(\tau_{n,1})-M_n|^{2+\delta}\}.
	$$	
	Therefore
	\begin{equation}\label{rate}
		\Delta_n(x)=|\mathbf{F}_n(x) - \Phi(x)| = |E(\mathbf{F}_n^\tau(x) - \Phi(x))|\leq
		E|\Delta_n^\tau(x)| := I_1+I_2,	
	\end{equation}
	where
	$$
	I_1 = \int_{\{V_n<\var\}}|\Delta_n^\tau(x)|dP,\;\;\;
	I_2 = \int_{\{V_n\geq\var\}}|\Delta_n^\tau(x)|dP.
	$$
	Evidentely
	$$
	I_1 \leq P\{V_n < \var\}.
	$$
	Consider $I_2.$ As
	$$
	E_\tau\{|X(\tau_{n,1})-M_n|^{2+\delta}\}= \frac{1}{\la(T_n)}\int_{T_n}|X(t)-M_n|^{2+\delta}\la(dt)\leq
	$$
	$$
	\hspace{100pt}2^{1+\delta}\left\{\frac{1}{\la(T_n)}\int_{T_n}
	|X(t)|^{2+\delta}\la(dt)+ |M_n|^{2+\delta}\right\},
	$$
	we have
	$$
	I_2 \leq \frac{C}{k_n^{\delta/2}}\var^{-\frac{2+\delta}{2}}  2^{1+\delta}\left[E |X(0)|^{2+\delta}+ E |M_n|^{2+\delta}\right].
	$$
	Since $E |M_n|^{2+\delta} \leq \E |X(0)|^{2+\delta},$
	we get the estimation
	$$
	I_2 \leq C \frac{ 2^{2+\delta}}{k_n^{\delta/2}}\var^{-\frac{2+\delta}{2}}  E |X(0)|^{2+\delta},
	$$
	which gives finally (\ref{rate1}).
\end{proof}

\subsection{Invariance principle}\label{IP}
The next natural step is to state the functional central limit theorem (invariance principle).

Let $X(t,\omega),t\in \R^m$, be a    bi-measurable random field  over a probability space  $(\Omega_X,\F_X,P_X).$
Let $\{T_n\}$ be a sequence of increasing convex sets in $\R^m$ containing balls with radii $r_n\to\infty.$
Let, for each $n\in\N$,\\ $\tau_{n,i},\;i=1,...,k_n,$ be random variables over a probability space $(\Omega_\tau,\F_\tau,P\tau)$, each being uniformly distributed on $T_n$ and independent of each other and of $X$.
Let  $(\Omega_{X,\tau},\F_{X,\tau},P_{X,\tau}):=(\Omega_X\times \Omega_\tau,\F_X\times \F_\tau,P_X\times P_\tau).$

Using random variables $X(\tau_{n,i})$ , we construct in usual way the continuous piecewise lineal random process $Z_n= \{Z_n(t),\;t\in [0,1]\}.$ The process  $Z_n$ has the vertices at the points
\begin{equation}
	\left(\frac{j}{k_n},\;\;\frac{S_j}{k_n^{1/2}V_n^{1/2}} \right),\;\;\;j=0,1,\ldots,k_n,
\end{equation}
where
$$
S_j = \sum_{r=1}^j (X(\tau_{n,r})-M_n),\;\;\; V_n = Var_\tau(X(\tau_{n,r})).
$$
Denote by ${\cal P}_n$ the distribution of the process  $Z_n$ in the space $\C[0,1]$ and by $W$ the distribution of standard Wiener process.
\begin{theorem}\label{pi}
	Under mentioned conditions:
	
	a) For $P_X$-almost all $\omega$
	\begin{equation}\label{pi1}
		Z_n(\omega,\cdot) \overset {P_\tau}\Longrightarrow W.
		\end{equation}
	b) The convergence
\begin{equation}\label{pi2}
Z_n \overset {P_{X,\tau}} \Longrightarrow W.
\end{equation}
also takes place.
\end{theorem}
\begin{remark}
	The importance of weak convergence in (\ref{pi1}, \ref{pi2}) comes from the fact that by
	continuous mapping theorem (see e.g., \cite{B}) many corollaries emerge
	immediately, namely for each $W$-a.e. continuous mapping
	$f:\C \to \R$ we have the convergence $f(Z_n) \Longrightarrow f(W)$. Tipical examples of such functionals are
	$f(x)=\sup_{t\in [0,1]}x(t),\;\; f(x)=\sup_{t\in [0,1]}|x(t)|$,
	f(x) = $\int_{[0,1]}h(t)dt,$ and so on...
\end{remark}

\begin{proof}
As for triangular array $\{X(\tau_{n,i})\}$	the Lindeberg condition is fulfilled, by Prokhorov's theorem
(see \cite{Pr}, Th.3.1.), we get (\ref{pi1}).

The relation (\ref{pi2}) follows from (\ref{pi1}) since the limiting measure does not depend of $\omega.$
\end{proof}
Let us consider now the random broken lines $\tilde{Z}_n$ with a different norma\-li\-zation.
More exactly, let
$$
\tilde{Z}_n\left(\frac{j}{k_n}\right)= \frac{S_j}{k_n^{1/2}\sigma^{1/2}},
$$
where $\sigma = Var\{X(0)\}.$

\begin{theorem}\label{pi3}
	Suppose that the field $X(t)$ is ergodic. Then:
	
	a) For $P_X$-almost all $\omega$
	\begin{equation}\label{pi41}
		\tilde{Z}_n(\omega,\cdot) \overset {P_{\tau}} \Longrightarrow W.
	\end{equation}
	b) The convergence
	\begin{equation}\label{pi42}
			\tilde{Z}_n \overset {P_{X,\tau}}  \Longrightarrow W.
	\end{equation}
	also takes place.
\end{theorem}
\begin{proof}
	The proof follows from the relation  $\tilde{Z}_n(t)= (\frac{V_n}{\sigma})^{1/2}Z_n(t)$ and the fact that
	due to the ergodic theorem $\frac{V_n}{\sigma}\rightarrow 1$ a.s.
\end{proof}

With a little additional restriction, we can replace $M_n$ by $\mu$ which is the mean value of our field.

 Define  the process $\tilde{\tilde{Z}}_n$ as a broken line with vertices
 at the points
 $$
 \left(\frac{j}{k_n},\; \frac{S_j}{k_n^{1/2}\sigma^{1/2}}\right),
 $$
where  now $S_j =\sum_{r=1}^j (X(\tau_{n,r}-\mu)).$
\begin{theorem}\label{pi4}
	Suppose that the field $X(t)$ is ergodic.
Additionally suppose that a.s.\footnote{See \S\ref{remarks} for remarks on this condition.}
$$
k_n^{1/2}(M_n - \mu) \rightarrow 0.
$$	
	Then:
	
	a) For $P_X$-almost all $\omega$
	\begin{equation}\label{pi43}
		\tilde{\tilde{Z}}_n(\omega,\cdot) \overset {P_{\tau}} \Longrightarrow W.
	\end{equation}
	b) The convergence
	\begin{equation}\label{pi44}
		\tilde{\tilde{Z}}_n \overset {P_{X,\tau}}  \Longrightarrow W.
	\end{equation}
	also takes place.
\end{theorem}
\begin{proof}
	It is sufficient to remark that a.s.
	$$
	\sup_{t\in [0,1]}|\tilde{\tilde{Z}}_n(t) - \tilde{Z}_n(t)| \leq
	\sigma^{-1/2}(k_n^{1/2}(M_n-\mu)) \rightarrow 0.
	$$
\end{proof}
%%%%%%%%%%%%%%%%%%%%%%%%%%%%%%%%%%%%%%%%%%%%%%%%%%%
\section{Empirical distributions}\label{ED}
\subsection{Randomized  versions of  the Glivenko-Cantelli theorem}
In this section we generalize the Glivenko-Cantelli theorem to multidimensional distributions of  stationary random processes on $\R$  and of  homogeneous random fields on $\R^m$ (to simplify the presentation we assume ergodicity; this assumption can be readily dropped by the consideration of  the conditional expectations as in Lemma \ref{L} below).    Generalized   versions of the Glivenko-Cantelli theorem for marginal distributions of stationary random sequences were proved in  \cite{SS,T}.
We consider a randomized version of such theorem for ergodic homogeneous random fields on $\R^m$.

Let $\mathbf{X}(t,\omega)=(X_1(t,\omega),...,X_l(t,\omega))$ be an $l$-dimensional bi-measurable random field on $\R^m$ over a probability space  $(\Omega_X,\F_X,P_X)$.

We consider the CDF of this random field:
$$F(x_1,...,x_s):= P(X_1(0)\le x_1,...,X_l(0)\le x_l).$$

Let $\{T_n\}$ be a sequence of increasing convex sets in $\R^m$ containing balls with radii $r_n\to\infty$ and let, for each $n\in\N$,\; $\eta_{n,i},\;i=1,...,k_n,$ be random variables over a probability space $(\Omega_\eta,\F_\eta,P_\eta)$, each being uniformly distributed on $T_n$ and independent of each other and of $X$. Denote: $(\Omega_{X,\eta},\F_{X,\eta},P_{X,\eta}):=(\Omega_X\times \Omega_\eta,\F_X\times \F_\eta,P_X\times P_\eta).$
\begin{lemma}\label{L}  If $Y(t,\omega),\;t\in\R^m$, is an  homogeneous   random field  over the probability space  $(\Omega_X,\F_X,P_X)$ and  for some  $\delta>0$
	\begin{equation}\label{cond}
		E_X[|  Y(0)|^{2+\delta}] <\infty,
	\end{equation}
	then with $P_{X,\eta}$-probability 1
	\begin{equation*}\label{}
		\lim_{n\to\infty}\frac1{k_n}\sum_{i=1}^{k_n} Y (\eta_{n,i},\omega) = E_X[Y(0)|\mathcal{I}_X)]
	\end{equation*}
	(if the random field $X$ is ergodic, $ E[Y(0)|\mathcal{I}_X)]= E_X[Y(0)])$.
\end{lemma}
\begin{proof} Denote $M_n[Y]:= \frac1{\la(T_n)}\int_{T_n}Y(t)\la(dt)$. Since with $P_X$-probability 1 for each $n\in\N$ the random variables $ Y (\eta_{n,i},\omega)),\;i=1,...,k_n$ are independent, condition \eqref{cond} implies:
	$P_X$-a.s. with  $P_\eta$-probability 1
	\begin{equation}\label{CLT}
		\lim_{n\to\infty}  \frac1{k_n}\sum_{i=1}^{k_n} Y(\eta_{n,i},\omega) - M_n[Y]]=0,
	\end{equation}
	by virtue of Corollary 1 (with p=1) and Remark 3 in \cite{HMT}. The Fubini theorem implies that \eqref{CLT} is valid with $P_{X,\eta}$-probability 1. By the PET, with $P_X$-probability 1
	\begin{gather}
		\lim_{n\to\infty} M_n[Y] =E[Y(0)|\mathcal{I}_Y].%\nonumber
	\end{gather}
	It remains to note that
	\begin{gather*}
		\frac1{k_n}\sum_{i=1}^{k_n}[ Y (\eta_{n,i},\omega) - E[Y(0)|\mathcal{I}_Y)]=\nonumber\\
		\left(\frac1{k_n}\sum_{i=1}^{k_n}Y (\eta_{n,i},\omega) - M_n[Y]\right)+( M_n-E[Y(0)|\mathcal{I}_Y]).\label{repr}
	\end{gather*}
\end{proof}

%In what follows  $Ind[A]$ means  the indicator of the set $A$.
Denote:
\begin{gather*} F_n (   x^1,...,x^l;\omega,\eta):=\\ \frac1{k_n}\sum_{i=1}^{k_n}Ind[X_1( \eta_{n,i},\omega) \le x^1,..., X_l(  \eta_{n,i},\omega) \le x^l].
\end{gather*}

To simplify the notation we  denote: $   \x =( x^1,...,x^l)\; ( \in \R^l)$.
\begin{theorem}
	Let the random field $\mathbf{X}(t,\omega),\;t\in R^m,$ be ergodic and homogeneous.   With $P_{X,\eta}$-probability 1
	\begin{equation}
		\lim_{n\to\infty} \underset{x\in\R^l} \sup |F_n( \x; \omega,\eta  )-F( \x)|=0.
	\end{equation}
	
\end{theorem}

\begin{proof}
	Since the random field $X$ is bi-measurable, the function $ (\omega,\eta)\mapsto  F_n( \x;\omega,\eta  )$  is $\F_{X,\eta}$-measurable. Let $\mathbb Q$ be the set of rational numbers. It is clear that the function      $(\omega,\eta)\mapsto \underset{\x\in\mathbb Q^l} \sup |F_n(\x; \omega,\eta  )-F( \x)|$  is $\F_{X,\eta}$-measurable, and, since $\mathbb Q^l  $ is dense in $\R^l$,
	$$\underset{\x\in\R^l} \sup |F_n( \x; \omega,\eta  )-F( \x)|= \underset{\x\in\mathbb Q^l} \sup |F_n( \x; \omega,\eta  )-F( \x)|,$$
	the function $\underset{\x\in\R^l} \sup |F_n( \x; \omega,\eta  )-F( \x)|$ is also $\F_{X,\eta}$-measurable.
	
	%	The transformation  $y^i=\frac1\pi \arctan x^i+\frac12 ,\;i=1,...,k,$ brings us to consideration  of $\x\in\R^l$ instead of   $\x\in [-\infty,\infty)^l$.  %In what follows the set
	%We consider the following  partial order in $R^{ l}$: $x\le y$ if $x^i\le y^i,i=1,...,l$.

	For each $k,\; 0\le k\le M-1,$ we denote:
	$$
	L_{k,M}= \{\x\in\R^l: \frac kM\le F(\x)< \frac  {k+1}M\}.
	$$
	Of course,  if $F$  is continuous, then   each $\y\in\R^l $ has at least one preimage with respect to the function $F$ and, for each $k$, $F(L_{k,M})=[\frac kM,\frac {k+1}M)$; otherwise, some points in $[0,1] $ do not have   preimages, and, therefore, there are  $k$, such that $[\frac kM,\frac {k+1}M)\setminus F(L_{k,M})\ne\emptyset$,
	and even such that $L_{k,M}=\emptyset$.
	If $\frac kM $ possesses at least one  preimage with respect to $F(\x)$, we denote by  $\x_{k,M}$  one of these preimages; in general,  if $
	L_{k,M}\neq\emptyset $, consider the  set $D_{k,M}:=F(L_{k,M})$, i.e., the set  of all points $\y$ in $[\frac kM,\frac{k+1}M)$ possessing preimages.
	
	Let $0\le k\le { M-1}$ and $L_{k,M}\ne\emptyset$.
	
	Denote:
	\begin{gather*}
		G(k+1,M):=P\{(X_1(0) ,...,X_l(0))\in F^{-1}([0,\frac{k+1}M))\},\\
		H(k,M):=1-P\{(X_1(0) ,...,X_l(0))\in F^{-1}([ \frac{k }M,1])\},
	\end{gather*}		
	and
	\begin{gather*}		
		G_n(k+1,M);\omega,\eta):  =\\
		\frac1{k_n}\sum_{i=1}^{k_n}\text{Ind}[(X_1(\eta_{n,i},\omega),.., X_l(\eta_{n,i},\omega)) \in F^{-1}([0,\frac{k+1}M)],\\
		H_n(k,M);\omega,\eta):  =\\
		1-	\frac1{k_n}\sum_{i=1}^{k_n}\text{Ind}[(X_1(\eta_{n,i},\omega),.., X_l(\eta_{n,i},\omega)) \in F^{-1}([ \frac{k}M,1)]].
	\end{gather*}
	
	Note: a) if $l=1 $, then $G(k,M)=F(\x _{k+1,{M}}-0); H(k,M)=F(\x_{k,M}+0)$;
	
	b) if  the function $F$ is continuous at the point $\x _{k+1,M}$,  then $G(k+1,M) = F(\x_{k+1,M}); H(k,M)=F(\x_{k,M})$.
	
	c) if $\x\in F(L_{k,{M}})$, then $ H(k,M)\le F(\x)\le G(k+1,M)$;\\ $H_n(k,M);\omega,\eta)\le F_n(\x;\omega\eta)\le G_n(k+1,M;\omega,\eta)$;
	
	d)  $0\le G(k+1,M)-H(k,{M})\le \frac1{M}$.

	% $ F(x) -F(y))<\frac1M$ if $x,y\in L_{k,M},k=0,...,M-1$.
	%Denote: $ L_{k,M}:=\{x\in [0,1)^s: F(x)\le \frac  kM\}, \overset{o}L_{k,M}-0:=\{x\in \R^s: F(x)< \frac  kM\}$.
	
	%Denote by $D{k,M}$ set of discontinuity points $x\in [0,1)^s$ of the function $F(\cdot  )$  such  that $\frac kM\notin $[F(x-0).
	%If $l=1$  %\etae put $t=0$ \etarite $F(x)$ instead if $F(0;x)$.  If the equation $F(x)=\frac kM$ has at least one solution \etae denote by $x_{k,M}$ any of them and put  $ L_{k,M}=\{x\in [0,1)^s: \frac kM\le F (x)< \frac  {k+1}M\},\ (k=0,...,M-1)$ .
	%  and $F(t;x)$ suffers a jump  at   $x_0\in [0,1)$, and  $F(t; x_0-0)\le \frac kM<...<\frac {k+m}M < F(t;x_0) $, but $\frac {k-1}M<F(t;x_0-0),\frac  {k+m+1}M>F(x_0)$\ $(m\ge 0)$;  clearly,
	% $L_{k,M}=...=L_{k+m-1,M}=\emptyset$ and $ L_{k-1 ,M}=\{x\in [0,1)^s: \frac {k-1}M\le F(x)< F(x_0-0)\},L_{k+m ,M}=\{x\in [0,1)^s:    F(x_0)\le F(x)< \frac  {k+1}M\}$; in general, if $L_{k,M}\neq \emptyset$, then $x_{k,M}:=\min L_{k,M}\cap D$. %define $x_{k,M}:= \min \{x\in[0,1):\frac kM\le F(t;x)\},k=$. In this case
	
	% If $l\ge 2$,  define  $x_{k,M}$ as some continuity point of the function $F(\cdot  )$, belonging to the set $B_{k,M}:=\{x\in   [0,1)^s: F(x-0)= \frac  kM\}$;

	%If $l\ge 2$,   $x_{k,M}$ is some  point in $L_{k,M}$ such that  $F(x_{k,M})=  \frac kM$.
	
	By  Lemma \ref{L}, for each $k$
	with $P_{X,\eta}$-probability 1
	\begin{gather}\label{lim2}
		|H_n (k,M;\omega,\eta)- H (k,M)| \to0
	\end{gather}
	and
	\begin{equation}\label{plim2}
		|G_n (k+1,M;\omega,\eta)- G (k+1,M)|\to0.
	\end{equation}
	If $\x\in L_{k,M}$,
	\begin{gather*}
		%F_n(\mathbf{t}x;(\omega, \eta))-F(\mathbf{t}; x)\le
		F_n(\x;\omega,\eta) -F (\x)\le  G_n (k+1,M; \omega )-H(k,M)=\\
		\left(G_n (k+1,M;\omega,\eta)-G ( k+1,M )\right)+
		\left(G(k+1,M) -H(k,M) \right)\le\\
		G_n (k+1,M;\omega,\eta)- G (k+1,M)  +\frac1{M}.
	\end{gather*}
	Similarly,
	\begin{gather*} F_n(\x; \omega, \eta )-F( \x)\ge
		H_n(k,M; \omega,\eta )-H({k,M} ) -\frac1{M}.
	\end{gather*}

	Therefore, if $0\le k\le {[M]} $, $L_{k,M}\ne\emptyset$ and $\x\in L_{k,M}$,
	\begin{gather*} | F_n(\x;\omega,\eta)-F( \x)|\le\\
		\max\{|G_n( k+1 ,M  );\omega,\eta)-G( k+1,M )|,
		|H_n( k,M; \omega,\eta )-H({k,M})|\}+\frac1{M},
	\end{gather*}
	It is clear that $\cup_{\{0\le k\le {M}-1\}}L_{k ,M}=\R^l$; therefore,
	\begin{gather*}\label{ln}
		\underset {\x\in\R^l}\sup| F_n(\x;\omega,\eta)-F( \x)|\le \\
		\max\{|G_n(k+1,M );\omega,\eta)-G(k+1,M)|,
		|H_n( k,M; \omega,\eta )-H({k,M})|\}+\frac1{M},
	\end{gather*}
	and, by relations \eqref{lim2} and \eqref{plim2}, with $P_{X,\eta}$-probability 1
	%$$ 0\le\limsup_{n\to\infty}\underset {x\in L_{k,M}}\sup  | F_n(x;\omega, \eta)-F( x)|\le \frac1{M}$$
	
	\begin{gather*}
		0\le \limsup_{n\to\infty}  \underset {\x\in\R^l}\sup| F_n(\x;\omega,\eta)-F( \x)|\le \\
		\limsup_{n\to\infty} \underset {k:L_{k ,M}\ne\emptyset,0\le k\le {M}-1 }\max \underset {\x\in L_{k ,M}}\sup| F_n(\x;\omega,\eta)-F( \x)|\le \frac1{M}.
	\end{gather*}

	Since $M$ is chosen arbitrarily, with $P_{X,\eta}$-probability 1
	$$\lim_{n\to\infty}\underset {\x\in\R^l}\sup | F_n(\x;\omega, \eta)-F( \x)|=0. $$
	
\end{proof}
%we have:
%$$ \limsup_{n\to\infty}\underset {x\in (0,1]^l}\sup  F_n(x;(\omega, w))-F(; x)\le \frac1M$$
%and
%$$ \liminf_{N\to\infty}\underset {x\in (0,1]^l}\inf F_n(x;(\omega, w))-F(; x)\ge -\frac1M$$

Let $X(t,\omega),\,t\in \R^m$, be a   random field  over a probability space \\ $(\Omega_X,\F_X,P_X)$. 
%We fix a set  of "time" points
% $\textbf{t}=\{t_1,...,t_l\in \R^m\}$ where  $l\in\N$, and study the %multidimensional CDF
% $$ F(\textbf{t}; x^1,...,x^l):=P(X(t_1)\le x^1,...,X(t_l)\le x^l).$$
%\begin{remark}
%	In particular, taking $l=1, t_1=0, x\in \R,$
%	and denoting by $F_n$ the one-dimensional marginal empirical distribution function:
%$$
%	F_n(x;\omega,\eta ):=
%	\frac1{k_n}\sum_{i=1}^{k_n}\1_{[X(\eta_{n,i},\;\omega,\eta) \le x]},
%	$$
%we deduce from (\ref{GL1})  the convergence with $P_{X,\eta}$-probability 1
%\begin{equation}\label{GL2}
%	\lim_{n\to\infty}\underset{x\in\R} \sup |F_n(x;\omega,\eta)-F(x)|=0,		
%\end{equation}		
%	where $F$ is the distribution function of $X(0).$		
%\end{remark}

We fix a set  of "time" points
$\t= \{t_1,...,t_l\in \R^m\}$ where  $l\in\N$, and study the multidimensional CDF
$$ F(\t; x^1,...,x^l):=P(X(t_1)\le x^1,...,X(t_l)\le x^l).$$

Denote:
\begin{gather*} F_n(\t; x^1,...,x^l;\omega,\eta):=\\ \frac1{k_n}\sum_{i=1}^{k_n}Ind[X(t_1+\eta_{n,i},\omega) \le x^1,..., X(t_l+\eta_{n,i},\omega) \le x^l];
\end{gather*}
It is the $l$-dimensional randomized empirical distribution function, based on observations of $X$ on randomly chosen points in $\cup_{j=1}^l(t_j+T_n)$.

%To simplify the notation we  denote: $   x =( x^1,...,x^l)( \in \R^l)$ and drop the fixed "time"   values $t_1,...,t_l$ in the notation of functions.

We apply the latter theorem to the ergodic homogeneous random field $\mathbf X(t,\omega)=(X(t_1+t),...,X(t_l+t)$ and come to the following statement.
\begin{theorem}\label{timeCDF}
	Let the random field $X(\t,\omega),\;\t\in R^m,$ be homogeneous and ergodic.   With $P_{X,\eta}$-probability 1
	\begin{equation}
		\lim_{n\to\infty} \underset{ (x^1,...,x^l) \in\R^l} \sup |F_n(\t; x^1,...,x^l;\omega,\eta)- F ( \t;x^1,...,x^l; \omega,\eta  ) |=0.
	\end{equation}
\end{theorem}
Let $\eta^u=\{\eta_{n,i}^u\},u=1,...,r$  be $r$ independent arrays of randomizing random variables. 
\begin{gather*}
	F_n(\t; x^1,...,x^l;\omega,\eta^u):= \\
	\frac1{k_n}\sum_{i=1}^{k_n}Ind[\{(\omega,\eta):X(t_1+\eta_{n,i}^u,\omega) \le x^1,..., X(t_l+\eta_{n,i}^u,\omega) \le x^l\}]; \\
	F_n^r(\t; x^1,...,x^l;\omega,\eta):=\frac1r\sum_{u=1}^r  F_n(\t; x^1,...,x^l;\omega,\eta^u).
\end{gather*}
\begin{corollary}
	Let the random field $\mathbf{X}(\t,\omega),\;\t\in R^m,$ be ergodic and homogeneous.   With $P_{X,\eta}$-probability 1
	\begin{equation}
		\lim_{n\to\infty} \underset{x\in\R^l} \sup |F_n^r( \t; x^1,...,x^l;\omega,\eta )-F( x)|=0.
	\end{equation}
	\end {corollary}
	This follows from Theorem \ref{timeCDF} and the relation:
	\begin{gather*}
		\underset{ (x^1,...,x^l) \in\R^l}\sup |F_n^r( \t  ; x^1,...,x^l; \omega,\eta  )-F( x)|\le\\ \frac1r\sum_{u=1}^r \underset{ (x^1,...,x^l) \in\R^l}\sup |F_n(\t; x^1,...,x^l;\omega,\eta^u)-F( x)|.
	\end{gather*}
	If $r\ge2$, each  estimator $F_n^r(\t; \x; \omega,\eta  )$ of the CDF involves more observations and, therefore, it is more precise than the estimator $F_n(\t;\x;\omega,\eta) $.
%%%%%%%%%%%%%%%%%%%%%%%%%%%%%%%%%%%%%%%%%%%%%%%%%%%

\subsection{Convergence of the distributions of empirical processes}
 We suppose in this section that $\{X(t,\omega),\,t\in \R^m\},$ is a    bi-measurable stationary random field  with values in $\R^l.$
Moreover we suppose that a.s. $X(t) \in (0,1)^l,$ as the general case can be
reduced to this by the standard transformation
$y^i=\frac1\pi \arctan x^i+\frac12 ,\;i=1,...,l.$

Let
\begin{equation}\label{emp1}
G_n(\x):= {k_n^{1/2}}(F_n(\x)-F(\x)),\;\;\;\x\in [0,1]^l,	
\end{equation}
where $F_n$ is the empirical distribution function for the series of r.v.
\\$\{X(\tau_{n,i}), i = 1,2,\ldots,k_n\},$
$$
F_n(\x) = \frac{1}{k_n}\sum_{i=1}^{k_n}\1_{[0,\x]}(X(\tau_{n,i})),
$$
where $[0,\x] = \prod_{i=1}^{l}[0,x_i]$
and $F$ is the distribution function of $X(0).$  Below we suppose that $F$ is continuous.

We will use also empirical processes with different centering:
$$
L_n(\x) = {k_n^{1/2}}(F_n(\x)-M_n(\x)),\;\;\;\x\in [0,1]^l,	
$$
where $M_n(\x)= \frac{1}{\lambda(T_n)}\int_{T_n}\1_{[0,\x]}(X(s))\lambda(ds).$

Let  $B([0,1]^l)$
be the set of bounded, real-valued, and measurable functions defined
on the $l$-dimensional cube $[0,1]^l$. Let
$ C([0,1]^l)$ be the set of all continuous functions.
Furthermore, let $\DS$ be a subset of
$ B([0,1]^l)$  such that, first,
it contains  $ C([0,1]^l)$ and, second,
the supremum norm of every function in
$\DS$ is determined by the supremum of
the function over a countable subset of  $[0,1]^l$.
Furthermore, assume that there is a metric $\rho$ that
makes $\DS$ a complete separable space whose topology
is weaker than the topology of uniform convergence, and such that
each time when $\rho(x_n, x)\rightarrow 0$ and $x$ is continuous it follows that $x_n$ converges to $x$ uniformly.
Examples of the space  $\DS$ with various
metrics/topologies with mentioned properties
%that make it complete and separable
can be found in a number of works (cf., e.g., \cite{St}, \cite{Dud}
and references therein). For comments on weak convergence
in non-separable spaces
%and how the results of the present
%paper can be adjusted to become valid beyond
%the Skorokhod space $\DS$,
we refer to Section 4 of \cite{DZ} .

It is clear that we can consider $G_n, L_n$ as  random elements of $\DS.$

 We  use the notation
  $\mathring{W}_F$ for the continuous Gaussian  centered process	with correlation function
  $K(\x,\y)=F(\x\wedge \y)-F(\x)F(\y)),\;  \x, \y \in [0,1]^l,$ where
    $\x\wedge \y =(x_1\wedge y_1,\ldots, x_l\wedge y_l).$
    % For the case $F(\x)=\x$ we use the notation $\mathring{W}.$

 Our aim is to state the weak convergence of $G_n$ and $L_n$ to $\mathring{W}_F.$
\begin{theorem}\label{BB}
If the field $\{X_t\}$	is ergodic, then
\begin{equation}\label{BB1}
	L_n \Longrightarrow  \mathring{W}_F.
\end{equation}
If additionally
\begin{equation}\label{cond2}
	\sqrt{k_n}\sup_{\x\in [0,1]^l}|M_n(\x) - F(\x)|   \overset P\to 0,
\end{equation}
then
\begin{equation}\label{BB2}
	G_n \Longrightarrow  \mathring{W}_F.
\end{equation}
\end{theorem}

\begin{remark}

Condition \eqref{cond2} imposes a restriction on the rate of growth of the sequence $k_n$. In the case, when the a.s. convergence is considered in this condition,   the admissible growth of the integers $k_n$ has been studied in \cite{DS} for  iid random sequences $X(1),X(2),...$; in \cite{Yu}, stationary sequences $X$ are studied, and  it is established,  how  the admissible rate of increase of $k_n$ is specified    by the rate of mixing of $X$ and by rate of growth of its   metric entropy. Condition \eqref{cond2} encourages to study the   admissible rate of the growth of $k_n$ when convergence in probability  is considered (we suspect it will be higher).

   We have
$$
\sqrt{k_n}\sup_{\x\in [0,1]^l}|M_n(\x) - F(\x)| =
\sqrt{\frac{k_n}{\lambda(T_n)}}\sup_{\x\in [0,1]^l}|\xi_n(\x)|,
$$
where
$$
\xi_n(\x)=\lambda(T_n)^{-1/2}\int_{T_n}[\1_{[0,\x]}(X(s))-F(\x)]
\lambda(ds).
$$
We see that $\xi_n$ is empirical process associated with initial process $X,$ hence,
in good cases \footnote{See for example \cite{Yu} which contains a large bibliography concerning the convergence of empirical processes associated with strictly stationary ones.}
	 $\xi_n$ has a continuous limit $\xi$. Then, as the functional $h\rightarrow \sup_{\x\in [0,1]^l}h(\x)$ is a.e. continuous with respect to the distribution of $\xi,$
$\sup_{x\in [0,1]^l}|\xi_n(x)|$ is bounded in probability, and the condition $k_n = o(\lambda(T_n))$ will be sufficient for (\ref{cond2}).
\end{remark}

	We start the proof of the Theorem \ref{BB} with the following Lemma.
	
Let $Y_{n,j},\; j = 1,\ldots,k_n$	be array of i.i.d. $l$-dimensional random vectors in each row having a common distribution function $J_n. $  It is supposed that
$Y_{n,j}\in [0,1]^l$ a.s.

Consider the empirical process
$$
V_n(\x) = {k_n^{1/2}}(H_n(\x)-J_n(\x)),\;\;\;\x\in [0,1]^l,	
$$
where $H_n(\x)= \frac{1}{k_n}\sum_{i=1}^{k_n}\1_{[0,\x]}(Y_{n,i}).$

	\begin{lemma}\label{emp}
Suppose additionally that $J_n$ converges uniformly to some  continuous  distribution
function $J.$

Then
\begin{equation}\label{lem}
	V_n \Longrightarrow \mathring{W}_J.	
\end{equation}

	\end{lemma}
\begin{remark}
	The case when $J_n = J$ for all $n$ is well known, see, for example,
\cite{BW}, \cite{St} or \cite{D}. But this result for
triangular arrays
we could not find in the literature  and therefore we present here
detailed proof.
It is based on the approach proposed in \cite{DZ}.

\end{remark}
\begin{proof}
At the beginning we consider the case when all one-dimensional marginal distributions
of $Y_{n,j}$ are
uniform on $[0,1].$

Following \cite{DZ}, we need to carry out three steps: establish the convergence of finite-dimensional distributions,
check the moment condition of Theorem 1 of \cite{DZ} and, finally, estimate the modulus of continuity of the function $J_n.$	
	
{\bf Step 1.}	The convergence of finite-dimensional distributions follows in the standard way from Lindeberg condition and Cramer-Wold device.

{\bf Step 2.} Verification of the first condition of Th.1 from \cite{DZ}.

	We shall find it technically more
convenient to work with \\
$\|\x\|:= \max_{1\leq i \leq l} |x_i|.$

	We prove that for some $C>0$ and for all $\x,\y \in [0,1]^l$
	\begin{equation}
	E|V_n(\x) - V_n(\y)|^{2l+2} \leq C\|\x-\y\|^{l+1}\;\; \text{whenever} \;\;\|\x-\y\|\geq \frac{1}{k_n}.
	\end{equation}
Indeed, $V_n(\x) - V_n(\y)$ is a normalized sum of i.i.d. centered random variables
 $\xi_i := \1_{[0,\x]}(Y_{n,i}) - \1_{[0,\y]}(Y_{n,i}) -(J_n(\x)- J_n(\y)).$

Since the symmetric difference of parallelepipeds $[0,\x], [0,\y]$
can be partitioned into a disjoint union of at most $2^l$ parallelepipeds of the form $[\boldsymbol a,\boldsymbol b],$ it suffices to estimate the moments for the individual elements of this partition. Let
$$
\al_i = \1_{[\boldsymbol a,\boldsymbol b]}(Y_{n,i}),\;\;\;
p = P\{Y_{n,i}\in[\boldsymbol a,\boldsymbol b] \}.
$$
As  every
parallelepipede $[\boldsymbol a,\boldsymbol b]$ from our partition has at least one edge whose length is $|x_i-y_i|$ for some $i$, then	$p\leq |J_n(\boldsymbol a)-J_n(\boldsymbol b)|\leq \|\x-\y\|.$

Using inequality of Th.19, \cite{P}, ch.III, we have
\begin{equation}\label{mom}
	E\left|\sum_{1}^{m}(\al_i-p)\right|^{2l+2} \leq C(mE\al_1^{2l+2} + (mE\al_1^2)^{l+1})
	\leq C(mp+m^{l+1}p^{l+1}),
\end{equation}
	since  $E\al_1^{2l+2} = p(1-p)[(1-p)^{2l+1} +p^{2l+1}]\leq p,$ and $E\al_1^2= p(1-p)\leq p.$

Returning to $E|V_n(\x) - V_n(\y)|^{2l+2},$ we get for $m=k_n$ and some $C>0$ the estimation
$$
E|V_n(\x) - V_n(\y)|^{2l+2} \leq C\left( \frac{p}{k_n^{l}}+ p^{l+1}\right).
$$
The right-hand side does not exceed $C\|\x-\y\|^{l+1}$
since $p\leq \|\x-\y\|$  and $\frac{1}{k_n}\leq \|\x-\y\|.$

{\bf Step 3.} We need to estimate $\sqrt{k_n}|J_n(\x)-J_n(\y)|$ when
$\x,\,\y$ are two adjacent points of the lattice $\Gamma_n = \{\boldsymbol{j}/k_n\,|\; \boldsymbol{j}\in [0,k_n]^l\}.$
Due to the condition that marginal distributions of $Y_{n,i}$ are
uniform we get immediately
$$
\sqrt{k_n}\sup_{\x,\y\in \Gamma_n,\;\|\x-\y\|=1/k_n }|J_n(\x)-J_n(\y)| \leq l/\sqrt{k_n}
\rightarrow 0
$$
when $n\rightarrow \infty,$ which complets the proof in this case.

Now we consider the case when limiting function $J$ is continuous and
coordinatewise strictly increasing. We modify slightly the arguments from the second part of the proof of the  Th.16.4. \cite{B}.

For simplicity we use the notation $Y_n$  for $Y_{n,1}$ and  $Y$ for
a random vector having distribution function $J.$ Let $J_n^{(k)},\;
J^{(k)}$ be marginal distribution functions for respectively $Y_n$ and
$Y$:
$$
J_n^{(k)}(x_k)= P\{Y_n^{(k)}\leq x_k\},\;\;\;
J^{(k)}(x_k)= P\{Y^{(k)}\leq x_k\},\;\;x_k \in [0,1].
$$
Let $\psi_n,\,\psi:[0,1]^l \to [0,1]^l$	be define by
$$
\psi_n(\x)  = (J_n^{(1)}(x_1),\ldots,J_n^{(l)}(x_l)),\;\;
\psi(\x)  = (J^{(1)}(x_1),\ldots,J^{(l)}(x_l)).
$$

Now let  $Z_{n,i}^{(k)}= J_n^{(k)}(Y_{n,i}^{(k)}), \;k=1,\ldots,l,\;\; i=1,\ldots, k_n.$ Then
vectors $ Z_{n,i} = (Z_{n,i}^{(1)},\ldots,Z_{n,i}^{(l)} )$
are i.i.d. and have $[0,1]$-uniformly distributed coordinates.

Let  $Z_n = \psi_n(Y_n),\;\;\; Z= \psi(Y). $ If $\x_n \to \x,$ then
$\psi_n(\x_n)\to \psi(\x)$ and by Th.5.5.\cite{B} we get the convergence $Z_n \Rightarrow \psi(Y)$ wich gives the convergence of distribution functions $F_{Z_n} (\x)\to F_Z(\x),$ and this convergence is uniform as $F_Z$ is continuous. Hence we can apply to the triangular array $\{Z_{n,i}\}$ our previous result which sais that the empirical processes $U_n,$
\begin{equation}
	U_n(\x) = \sqrt{k_n}(A_n(\x)-F_{Z_n}(\x)),
	A_n(\x) =  \frac{1}{k_n}\sum_{i=1}^{k_n}\1_{[0,\x]}(Z_{n,i}),
\end{equation}
 converge weakly to $\mathring{W}_{F_Z}.$

 Now define two mappings $\va_n,\;\va$ inverse to $\psi_n,\;\psi:$
 $$
\va_n(\y)=(...,\va_n^{(k)}(y_k),...),\;\;\; \va(\y)=(...,\va^{(k)}(y_k),...),
 $$
 $$
 \va_n^{(k)}(s) = \inf\{t\,|\,s\leq J_n^{(k)}(t)\},\;\;
 \va^{(k)}(s) = \inf\{t\,|\,s\leq J^{(k)}(t)\}.
 $$
It is clear that the vectors $(Y_{n,1},\ldots,Y_{n,k_n})$ and $(\va_n(Z_{n,1}),\ldots,\va_n(Z_{n,k_n}))$ will have the same distribution.

Define  $f_n, f:\DS\to \DS$ respectively  by   $$
(f_ny)(\x)=y(J_n(\x));\;\;(fy)(\x) = y(J(\x))),\;\;y\in \DS,\; \x\in [0,1]^l.
$$
 If $y_n$ converges to $y$ in $\DS$ and $y\in \C[0,1]^l,$ then the convergence is uniform and
 $f_n(y_n)$ will converge to $f(y)$ uniformly, due to the uniforme convergence of $J_n$  to $J.$
 Hence, by Th.5.5 \cite{B},
 $$
 f_n(U_n) \Longrightarrow f({\mathring{W}_{F_Z}}),
 $$
which gives the result since we have the equalities in distribution\\
 $f_n(U_n) \overset {d } = V_n $  and $f(\mathring{W}_{F_Z}) \overset {d} =\mathring{W}_J. $

 To complete the proof of the lemma, it suffices
 bring into consideration a family $\{L_\theta\}$ of
 triangular arrays, for which distribution functions $J_\theta$
  are continuous, strictly increasing, and uniformly converge to
the distribution function $J_{\theta_0}:= J$ when $\theta \to \theta_0.$
Let $V_{n,\theta}$ be the empiric process associated with $L_\theta$
and $V_{n,\theta_0}=V_n.$
Then we have the following properties:
\begin{enumerate}
	\item {For each } \;$\theta,\;\;\;   V_{n,\theta} \Rightarrow \mathring{w}_{J_{\theta}},$  as $n\to\infty$ (by previous consideration.)
\item  For each $n,\;\; V_{n,\theta} \Rightarrow V_{n,\theta_0,}$ as
$\;\theta \to \theta_0$ (due to the convergence\\
 $J_\theta \to J_{\theta_0}.$)

\item	$\mathring{W}_{J_{\theta}}\Rightarrow \mathring{W}_{J_{\theta_0}},$ as
$\;\;\theta \to \theta_0$ (due to the convergence of covariance functions.)
\end{enumerate}
From these three properties it follows immediately the convergence\\
$V_n=V_{n,\theta_0} \Rightarrow \mathring{W}_{J_{\theta_0}}= \mathring{W}_{J}$  which finishes the proof.
\end{proof}

{\it Proof of Th.\ref{BB}.}
Remark that by Glivenko-Cantelli theorem \\
$M_n(\x)= \frac{1}{\lambda(T_n)}\int_{T_n}\1_{[0,\x]}(X(s))\lambda(ds).$
a. s. converges uniformly to $F.$
Therefore, applying previous Lemma \ref{emp} to the vectors
$X(\tau_{n,i})$ which are  i.i.d. conditionally given field $X,$ we get
the convergence :
$$
	L_n \overset {P_\tau} \Longrightarrow  \mathring{W}_F,
$$
from which the convergence (\ref{BB1}) follows in the usual way.

The convergence (\ref{BB2})  follows
immediately from (\ref{BB1}) due to  the condition (\ref{cond2}):
$$
\sup_{\x\in [0,1]^l}|L_n(\x) - G_n(\x)| =
\sqrt{k_n}\sup_{x\in [0,1]^l}|M_n(\x) - F(\x)|   \overset P\to 0.
$$
\hfill $\square$
%\end{theorem}

Let $l=1.$ Applying once more the continuous mapping theorem \cite{B},
we deduce from (\ref{BB1}) an analog of the famous Kolmogorov's
	result.
\begin{corollary}
Let $l=1,$ $X$ be an ergodic homogeneous random field and   condition \eqref{cond2} be fulfilled. Then 
$$
	P\{\sup_y|  {k_n^{1/2}}(F_n(y)-F(y))|\leq x\} \rightarrow
	\sum_{k=-\infty}^{\infty}(-1)^k e^{-2k^2x^2},\;\;x> 0.
	$$
\end{corollary}

\begin{remark}
	Theorem \ref{BB}  allows one to analyze the asymptotic behavior of empirical processes associated with finite-dimensional distributions of the original field.
	
	Let $X(t,\omega),\,t\in \R^m,\; X(t)\in \R^l,$ be a   random field  over a probability space  $(\Omega_X,\F_X,P_X).$
	We fixe the points $t_1,...,t_l$ and consider the distribution function
	$$
	F( \x):=P\{X(t_1)\le x^1,...,X(t_l)\le x^l\},\;\;\; \x=(x^1,...,x^l),
	$$
	of the vector $(X(t_1),...,X(t_l))$  and associated with it
	 the $l$-dimensional randomized empirical distribution function $F_n$,
	\begin{gather*} F_n( \x):= \frac1{k_n}\sum_{i=1}^{k_n}Ind[X(t_1+\tau_{n,i}) \le x^1,..., X(t_l+\tau_{n,i}) \le x^l];
	\end{gather*}
here, as before, $\{\tau_{n,i}\}$ is a randomizing triangular array.
Let
$$
M_n(\x) = E_\tau (F_n(\x)) = \frac{1}{\lambda(T_n)}\int_{T_n}
\1_{[0,\x]}(X(t_1+s)),...,X(t_l+s))\lambda(ds).
$$

	Consider the empirical processes:
	$$
	L_n(\x) = {k_n^{1/2}}(F_n(\x)-M_n(\x)),\;\;\;\x\in [0,1]^l,	
	$$
	$$
	G_n(\x) = {k_n^{1/2}}(F_n(\x)-F(\x)),\;\;\;\x\in [0,1]^l.
	$$

	We apply the Theorem \ref{BB} to the ergodic homogeneous random field $\mathbf X(t,\omega)=(X(t_1+t),...,X(t_l+t))$ and come to the following statement.
	\begin{theorem}
		Let the random field $X(\t,\omega),\;\t\in R^m,$ be homogeneous and ergodic.   With $P_{X,\tau}$-probability 1
		\begin{equation}
			L_n \Longrightarrow  \mathring{W}_F.
		\end{equation}
		If additionally
		\begin{equation}\label{cond3}
			\sqrt{k_n}\sup_{\x\in [0,1]^l}|M_n(\x) - F(\x)|   \overset P\to 0,
		\end{equation}
		then
		\begin{equation}
			G_n \Longrightarrow  \mathring{W}_F.
		\end{equation}
	\end{theorem}
	
\end{remark}
\vspace{35pt}	

\begin{appendix}
	\centerline{\textbf{APPENDIX}}
	\section{Comments on   conditions \eqref{cond16} and \eqref{cond17}}\label{remarks}

%%%%%%%%%%%%%%%%%%%%%%%%%%%%%%%%

In what follows we put $d=1$ for simplicity.
\begin{remark}\label{k_n} The proofs of Theorems \ref{CLT1} and \ref{CLT2} suggest  that, in these theorems, it is reasonable to use sequences $\{k_n\}$ growing to $\infty$ rather fast, i.e., to chose the "size of randomization" as large as possible. But, in Theorems \ref{CLT3} and \ref{CLT4}, conditions \eqref{cond16} and \eqref{cond17} put some restrictions to this growth. These conditions   are related to the speed of convergence in the Pointwise and   Mean Ergodic Theorems, respectively. In the following notes we provide  an idea of these conditions.   We start with well known results related to condition \eqref{cond16} in the case $T=\Z,\;\; T_n^l=\{1,...,n\}$.
	
	1. There are no universal  sequences $k_n$ satisfying condition \eqref{cond16}: for each non-decreasing sequence $k_n\to\infty$ there is a bounded ergodic  stationary  sequence $\{X_n\}$ such that  $\{k_n\}$ does not satisfy this condition \cite{Kr}.
	
	2.  The sequence $k_n=n^2$ does not satisfy condition \eqref{cond16} \cite{K1}.
	
	3. For each ergodic  stationary random field there is a  non-decreasing sequence of integers $k_n\to \infty$ that satisfies condition   \eqref{cond16} \cite{KP2}.
	
	4. Note 1 shows that the rate of convergence in the PET depends on the properties of the field $X$; about this relation see \cite{K1,KP1,KP2} and references therein.
	
	5. If
	$\sup_n (n^\gamma  Var_X(M_n)))<\infty$, where $\gamma>0$,  then \eqref{cond16} holds with $k_n=n^\al$, where $\al<\frac{\gamma}2$ (\cite{CL}, Proposition 1). See Notes 6 and 9 - 11 below for examples.
	
	Now we turn to condition \eqref{cond17}.
	
	6. By the Chebyshev inequality, condition
	\eqref{cond17} holds if \\ $k_n Var_X(M_n)\to 0$. If the field $X$ is  wide-sense ergodic,  then, by the Mean Ergodic Theorem,  $ Var_X(M_n)\to 0$ (see Remark \ref {MET}); therefore, in this case   \eqref{cond17} is always fulfilled if  $k_n$ grows  sufficiently slow:
	\begin{equation}\label{L2}
		k_n =o(  (Var_X(M_n))^{-1}).
	\end{equation}

	7.  If $\{X(k)\}$ is a mixing  stationary random sequence,  $E_X[(X(0))^2]<\infty$ and $ T_n^l=\{1,...,n\}$, then the sequence $k_n=n^2$  does not satisfy  condition  \eqref{L2}.
	
	8. Let $T=\R^m$ or $\Z^m, m\ge 1 $. Assume that the CLT holds in its classical form: $ (\la(T_n))^\frac12\sigma^{-1}(M_n-\mu) \overset D \to Z$, where $Z$ is the standard normal random variable. We have: $k_n^\frac12(M_n-\mu) =(\frac{k_n\sigma^2}{\la(T_n)})^\frac12 (\la(T_n))^\frac12\sigma^{-1}(M_n-\mu)$; therefore $k_n^\frac12(M_n-\mu)\overset D\to 0$, if and only if $k_n=o((\la(T_n)))$.
	A rather often case when the classical CLT fails is when $(\la(T_n))^\frac12(M_n-\mu) \overset D \Rightarrow0$, hence    condition \eqref{cond17} is fulfilled with   $k_n=O((\la(T_n)))$.

	9. Let $\{\xi_k\}$ be an  ergodic stationary Markov chain  with the   probability   state space $(\mathcal X,\mathcal A,m    )$ ($m$ the initial probability  measure). Let $Q(x,A)$ be the transition function,  and let $F$ be the (dense) set  in   $L^2(\mathcal X,\mathcal A,m)$ consisting of all functions $f: f(x)=g(x)-\int_{\mathcal X}g(y)Q(x,dy),g\in L^2( \mathcal X,\mathcal A,m) $. Then,  for all     $f\in F$   the following alternative holds: the stationary random sequence $X_k=f(\xi_k)$ either satisfies the CLT in the classical form or $Var[M_n]=O(n^{-1})$\cite{GL}; this alternative holds also for $f$ in a larger class of functions, which contains $F$ (see \cite{MW}). According to the previous note,  in both cases each sequence $k_n=o(n)$ satisfies condition \eqref{cond17}.
	
	10. Let $m\ge1$, $X(t), t \in \Z^m$ be an wide-sense stationary random sequence on $T$, satisfying Condition (A),
$ E[X(0)] = \mu$. Let $R(t)$ be  its covariance function.

Condition \eqref{cond17} holds if

-  $k_n=o(n^{  a})$ and $|R(t)|=O(|t|^{-a})$, $0<a< m$, or

-  $k_n=o((\frac{n^m}{\log n}))$ and $|R(t)|=O(|t|^{-m})$, or

 -  $k_n=o(n^{m})$ and  $|R(t)|=O(|t|^{-a})$,  $a>m$).

%d) $r_n=o(n^{\frac m{2p}})$ and $\int_T|R(t)|^p\la(dt)<\infty \ (1< p<\infty).$
	11. Let  $T=\R^m$ or $\Z^m, m\ge 1$,  and let  $T_n=(0,n]^m,n\in\N$; if $X$ is wide-sense stationary and possesses a  spectral density $\psi$, which is continuous at $0$, then   condition \eqref{cond16} holds if $k_n=o(n^m)$. % if and only if $\va(0)=0$
	This follows from the relation  $$n^mVar_X\left(\int_{T_n}X(t)\la(dt)\right) =(2\pi)^m\psi({0})+o(1)$$ (in   the cases $T=\Z$ and $T=\R$ these are, respectively,  Theorems 18.2.1 and  18.3.2 in \cite{IL}; if $T=\Z^m,m\ge 2,$  this   is  Theorem 3 in \cite{K3}; the case  $T=\R^m,\\ \;m\ge 2,$ can be treated  similarly).
	
	12. In the case $m=1$ various restrictions on   the spectrum and the correlation function implying a given rate of convergence in the Mean Ergodic Theorem, hence on restrictions  on the choice of $k_n$ in \eqref{cond17} are discussed in  \cite{K1,K2,KP1,KR}.
\end{remark}

\section{Randomization using non-uniform \\ distributions}\label{measures}
Let $X(t)=(X^1(t),...,X^d(t)),t\in T,$ be a $\B\times \F$-measurable   random field with  homogeneous components ($d\in \N$) and $E[|X(0)|^{2}]<\infty$ for some $\delta>0$.  For each $l=1,...,d, n\in\N,$ consider  a   probability Borel measure   $\{q_n^l\}$ on $T_n^l$, which is  the distribution     of the $T^l_n$-valued i.i.d. randomizing  random vectors $\tau^l_{n, i},i=1,...,k_n,$ introduced  in Subsection 1.5.  By the Fubini  theorem, for each measurable function $f$ on $\R$  such that $E_P[f(X(0))]<\infty $ we have: $  \int_{T^l_n}|f(X(t))|q_n^l(dt)<\infty$ a.s.,  and
\begin{equation*}\label{Ew1}
	E_\tau [f(X(\tau_{n,i}))]= \int_{T^l_n}f(X(t))q_n^l(dt), i=1,..., k_n.
\end{equation*}
In particular,
\begin{gather*}\label{E1}
	M_n^l:= E_\tau [X(\tau_{n,i}^l)]=\int_{T_n^l}X(t)q_n^l(dt); \\ V_n^l:= Var_\tau[X(\tau_{n,i}^l)]=\int_{T_n^l}(X(t)-M_n^l)^2q_n^l(dt).
\end{gather*}
Hence
\begin{gather*} E_\tau[\sum_{i=1}^{k_n}f(X(\tau_{n,i}))]=
	k_n\int_{T_n^l}f(X(t))q_n^l(dt).
\end{gather*}
\begin{definition}\label{sets}
	We say that the sequence of probability Borel  measures $\{q_n\}$ is \emph{pointwise averaging}  if for each  homogeneous random field $X$
	the following PET with the "weights" $q_n$  holds: if $E_X[|f(X(0))|]<\infty $ then $\lim_{n\to\infty}\int_{T}f(X(t))q_n(dt)=E_X[f(X(0))|\mathcal{I}_X]$ with probability 1; in the case, when the measures $q_n$ possess densities $\va_n$ with respect to $\la$, we say that the sequence $\{\va_n\} $ is \emph{pointwise averaging}.
\end{definition}

We provide a simple lemma that helps to construct pointwise averaging sequences in $\R^m$.
\begin{lemma} \label{dens} Let $\{\va_n\}$ be a sequence of densities on $\R^m$ with compact supports $S_n$ in $T$; if  there are positive numbers $a_n$ such that the  sets $T_n^l:=\{(x,y):x\in S_n,0\le y\le a_n\va_n(x)\}$ form a  pointwise averaging sequence in $\R^{m+1}$, then $\{\va_n\}$ is pointwise averaging.
\end{lemma}
\begin{proof} Apply Definition \ref{admis} to the random field $Y(t,s)\equiv X(t), t\in\R^m,s\in\R,$ (note that $\lambda(T_n^l)=a_n$, hence $\frac1{\la(T_n^l)}\int_{T_n^l}Y(t,s)\la(dtds)=\int_{S_n}X(t)\va_n(t)dt$; here $\la$ is the Lebesgue measure on $\R^{m+1})$.
\end{proof}
The following two examples are implied  by Lemma \ref{dens} and Examples \ref{convex} and \ref{homot} (see also   Corollaries  4.2 and 4.3  in Ch. 6 in \cite{T}).

\begin{example}Let $\{\va_n\}$ be a
	sequence of  densities on
	$\R^m$ which are  concave on their supports $S_n$ (the sets $S_n$ are compact and convex).  If  the
	sequence $\{S_n\} $ is increasing and   the sets $S_n$ contain balls of radii $r_n\to\infty$, then the
	sequence  $\{\va_n\}$ is pointwise averaging; see \S 5.3  in \cite{TS}.
\end{example}

\begin{example}
	Let $\{c_n\}$ be a sequence of positive numbers tending to $\infty$.
	If  $\va$ is a bounded density on $\R^m$ with a  compact support $S$ containing $0$, then  the sequence of \emph{rescaled} densities $\va_n(x)=c_n^{-m}\va(c_n^{-1}x)$ is pointwise averaging.% (see Proposition 5.2 in \cite{TS}).  %the support of $\va_n$ is $T_n^l=a_nS:=\{t:t=a_ns,s\in S\}$.
\end{example}
\begin{remark}
	When   $\va(t)>0$ for all $t\in \R^m$, some more conditions on $\va$ are needed, but the class of "good" rescaled densities is still rather wide (see  Proposition 5.3 in \cite{TS}); for example, if $\va$ is the density of a nondegenerate symmetric  normal distribution, then the sequence of rescaled densities is pointwise averaging.
\end{remark}
For more examples, related to averages by convolutions of probability measures, see the next section.

It is easy to see that \emph{ Theorems \ref{CLT1} - \ref{CLT5} hold  for all pointwise averaging sequences} $\{q_n^l\}$ (of course, in the proofs  $\int_{T_n^l}f(X(t))q_n^l(dt)$ has to be considered  instead of
$\frac1{\la(T_n^l)}\int_{T_n^l}f(X(t))\la(dt)$).

\section{Randomized CLT on general groups} In the above study the simplest  groups  $\R^m$ and $\Z^m$ were considered. But the results are  valid  for all groups $T$ on which the  PET    holds  with some   sequence of sets $\{T_n^l\}$ or  "weights" $\{q_n^l\}$. Then the  proof of the generalized  versions of Theorems \ref{CLT2} and \ref{CLT} may be repeated almost word for word.   There is a rather rich literature related to PETs on groups; see \cite{Li}, \cite{N}, \cite{T}, \cite{TS} and the bibliographical survey therein; the  PET with exponential rates of convergence on semisimple Lie groups is presented in \cite{MNS}. In locally compact topological groups it is natural to consider pointwise averaging sequences of sets or densities with respect to the Haar measure. If the   group is not locally compact,  there is no Haar measure on it, and one has to consider pointwise averaging with sequences of probability Borel measures $\{q_n\}$ as  "weights" (see \S \ref{measures}).
If  $T$ is a second  countable topological group    and  the smallest closed group containing the support of a probability Borel measure  $p$ is $T$, then the sequence $\{q_n:=\frac1n\sum_{k=1}^np^{*k}\} $ is pointwise averaging on $T$   (see Corollary 5.3 in Ch. 3, Proposition 1.1  in  Ch. 5  and Theorem 6.1 in Chapter 6 in \cite{T}).
\emph{ In all these cases the analogs of our CLTs for homogeneous random fields with $E[(X(0)^2]<\infty$ are valid.}

 If, in addition, $p$ is  symmetric (i.e. $p(A^{-1})=p(A)$ for each Borel set $A$ in $T$) we may put $q_n:=p^{*n}$ (see Note 6.4 in Chapter 6 in \cite{T} and the references therein); the PET with these weights is valid for a homogeneous field $X$, if $E_X[|X(0)|^{1+\delta}]<\infty \ \text{for some}\  \delta>0$; note that condition $$E_X[|X(0)|^{2+\delta}]<\infty \ \text{for some}\  \delta>0$$ guarantees that the PET holds also for the random field   $(X(t))^2$; this, in its turn, implies  the analog  of Lemma \ref{lindeberg} and,  hence,  the analogs of the rest results. Therefore, under this condition, the CLTs   hold with the   "weights" $p^{*n}$, too.
\vskip1cm

\end{appendix}


\begin{thebibliography}{2}

 \bibitem{BFGK}
J. Beran, Y. Feng, S. Ghosh, R. Kulik,  Long-memory processes, Probabilistic properties and statistical methods. Springer, Heidelberg (2013).

\bibitem{BW}
P.J. Bickel, M.J. Wichura, Convergence criteria for multiparameter
stochastic processes and some applications,
The Ann. of Math. Stat.,42, 5 (1971) 1656--1670.

\bibitem{B}
P. Billingsley,  Probability and measure.  Wiley, New York  (1986)

\bibitem{Bo}
E. Bolthausen, On the central limit theorem for stationary mixing random fields.  Ann. Probab. 10 (1982), 1047–1050.

\bibitem{Bra}

R.C. Bradley, A remark on the central limit question for dependent random variables;
J. Appl. Probability 17 (1980), 94-101.

\bibitem{Br}
R. C. Bradley,
Information regularity and the central limit question.
Rocky Mountain J. Math. 13 (1983), 77–97.

\bibitem{Br1}
 R. C. Bradley, Basic properties of strong mixing conditions, A survey and
some open questions, Update of, and a supplement to, the 1986 original.
Probab. Surv. 2 (2005), 107–144.

\bibitem{Br2}
 R.C. Bradley, Introduction to strong mixing conditions, Vol 1-3. Kendrick
Press, Huber sity, UT (2007).

\bibitem{Browder}
 F. Browder, On the iterations of  transformations  in noncompact minimal dynamical systems.  Proc. Amer Math.  Soc.  9, 773-780 (1958).




\bibitem{Bu}
 A. V. Bulinskiĭ, A statistical version of the central limit theorem for vector-valued random fields. (Russian) Mat. Zametki 76 (2004), no. 4, 490–501; English translation:  Math. Notes 76, 455–464 (2004).

\bibitem{BZ}
 A. V. Bulinski\u{i}, I. G. \v{Z}urbenko,  The central limit theorem for random fields. Dokl. Akad. Nauk SSSR  226, 23–25 (1976) (Russian); English translation: Soviet Math. Dokl. 17 (1976), 14--17.

\bibitem{BD}
R. Burton, M. Denker, On the central limit theorem for dynamical systems. Trans. Amer. Math. Soc. 302 (1987), 715–-726.

\bibitem{Chung2}
K.L. Chung,  Markov chains with stationary transition probabilities,  2nd ed.  Springer, Berlin (1967)

\bibitem{CL}
G. Cohen, M. Lin, Laws of large numbers with rates and the one-sided ergodic Hilbert transform. Illinois J. Math. 47 (2003), 997--1031.

\bibitem{CDV}
 Ch. Cuny, J. Dedecker, D. Volný.
A functional CLT for fields of commuting transformations via martingale approximation.
Zap. Nauchn. Sem. S.-Peterburg.  441,  239–262 (2015); reprinted in
J. Math. Sci. (N.Y.) 219,  765–781 (2016).

\bibitem{Davydov1}
Yu. A. Davydov,  On strong mixing property for Markov chains with a countable number of states, Soviet Math. Dokl. 10 (1969), 825--827.

\bibitem{Davydov2}
 Yu. A.  Davydov,  Mixing conditions for Markov chains, Theory Prob. Appl. 12 (1973), 312--328.

\bibitem{DZ}
Yu.  Davydov,  R. Zitikis,
 On weak convergence of random fields, AISM (2008) 60, 345–-365.


\bibitem{De} J. Dedecker, A central limit theorem for stationary random fields,   Probab. Theory Relat. Fields 110 (1998), 397--426.

\bibitem{Der}
 Y. Derriennic, Some aspects of recent works on limit theorems in ergodic theory with special emphasis on the ”central limit theorem”, Discrete Contin.
Dyn. Syst. 15 (2006), 143—158.

\bibitem{DL}
Y. Derriennic, M. Lin, Variance bounding Markov chains,
$L_2$-uniform mean ergodicity
and the CLT.  Stochastics and Dynamics,  11   (2011), 81–94


\bibitem{D} J. L. Doob, Stochastic processes. Reprint of the 1953 original.  John Wiley \& Sons, N.Y (1990).

\bibitem{Dud}
 R.~M. Dudley,
Uniform Central Limit Theorems.
Cambridge University Press, Cambridge
(1999).

\bibitem{DS}
N. Dunford, J.T. Schwartz,  Linear operators, Part I, General theory. Wiley, New York (1988).

\bibitem{GS}
P. Gaenssler,  W. Stute, Empirical processes: a survey of results for
 independent and identically distributed random
 variables,   Ann. Probability
7 (1979), 193-243

\bibitem{G}
L. Giraitis , H.L. Koul, D. Surgailis,  Large sample inference for long memory processes,  Imperial College Press, London (2012).

\bibitem{GL}
M.I. Gordin, B.A. Lif$\check{s}$ic, Central limit theorem for stationary Markov processes,  Dokl. Akad. Nauk SSSR 239 (1978), 766--767 (Russian). English translation: Soviet Math. Dokl. 19 (1978), 392--394.

\bibitem{Herrndorf}
N. Herrndorf,  Stationary strongly mixing sequences not satisfying the central limit theorem, Ann.  of Prob. 11 (1983), 808--813.

\bibitem{H}
O. H$\ddot{a}$ggstr$\ddot{o}$m, On the central limit theorem for geometrically
ergodic Markov chains. Probab. Theory Relat. Fields 132 (2005)  74–82.

\bibitem{HMT}
T.-C. Hu, F. M\'{o}ricz, R. L. Taylor, Strong laws of large numbers for arrays of
rowwise independent random variables. Acta Math. Hung. 54 (1--2) (1989), 153--162.

\bibitem{I}
I. A. Ibragimov, Some limit theorems for stationary processes. (Russian) Teor. Verojatnost. i Primenen. 7 (1962) 361–392. English translation: Theory of Probab. and Appl. 7 (1962), 349 - 382.

\bibitem{I2}
 I. A. Ibragimov,
A remark on the central limit theorem for dependent random variables.
Teor. Verojatnost. i Primenen. 20 (1975)  134–140 (Russian). English translation: Theory of Probab. and Appl. 20 (1975), 135 - 141.

\bibitem{IL}
 I.A. Ibragimov, Yu.V. Linnik,   Independent and stationary sequences of random variables. Wolters-Noordhoff Publishing, Groningen (1971).

\bibitem{IvLe}
 A. V. Ivanov, N. N. Leonenko,  Statistical analysis of random fields.
 Mathematics and its Applications (Soviet Series), 28. Kluwer Academic Publishers Group, Dordrecht  (1989).

\bibitem{K1}
A.G. Kachurovskii,   The rate of convergence in ergodic theorems, Uspekhi Mat. Nauk   51 (1996) 73-124 (Russian). English translation:  Russian Math. Surveys, 51 (1996)   653-703.

\bibitem{K3}
 A.G. Kachurovskii, Convergence of averages in the ergodic theorem for groups $\Z^d$. Zapiski Nauchnykh Seminarov POMI  256 (1999), 121--128  (Russian). English translation:   J. Math. Sci.   107 (2001) 4231--4236.

\bibitem{K2}
 A.G. Kachurovskii, The Fejer integrals and the von Neumann ergodic theorem with
continuous time, Zap. Nauchn. Sem. POMI  474 (2018) 171–182 (Russian). English translation: J. Math. Sci. 251 (2020), 111-118.

\bibitem{KP1}
 A.G. Kachurovskii, I.V. Podvigin,  Estimates of the rate of convergence in the von Neumann and Birkhoff ergodic theorems. Trudy Moskov. Mat. Obs. 77 (2016), 1-66 (Russian). English  translation: Trans. Moscow Math. Soc  1-53 (2016) .

\bibitem{KP2}
A.G. Kachurovskii,  I.V. Podvigin,  Measuring the rate  of convergence in the Birkhoff ergodic theorem. Matematicheskie  zametki 106 (2019)  40--52 (Russian); English translation:   Math. Notes 106 (2019)  52–62.

\bibitem{KR}
 A.G. Kachurovskii, A.V. Reshetenko, On the rate of convergence in von Neumann's ergodic theorem with continuous time. Mat. Sb. 201 (2010)  25–32 (Russian); English translation in Sb. Math 201 493--500 (2010) .

\bibitem{Kr}
U. Krengel, Ergodic theorems. De Gruyter, Berlin (1985).


\bibitem{L}
E.L. Lemann,  Elements of large-sample theory. Springer, New York (1999).

\bibitem{Le}
N.N. Leonenko, On the central limit theorem for certain classes of random fields. Teor.Verojatnost. i Mat. Statis., 15 (1976)  114--122.

\bibitem{Li}
E. Lindenstrauss, Pointwise  theorems for amenable groups. Invent. Math. 146 (2001), 259-295.

\bibitem{Lo}
M. Loève,  Probability theory.  Van Nostrand,   Princeton, N.J.-Toronto, Ont.-London (1963)


\bibitem{Ma}
M. El Machkouri,
Kahane-Khintchine inequalities and functional central limit theorem for stationary random fields.
Stochastic Process. Appl. 102,  285–299 (2002).


\bibitem{MVW}
M. El Machkouri, D. Volný, W.B. Wu, A central limit theorem for stationary random fields. Stochastic Process. Appl. 123 (2013) 1–14.

\bibitem{MNS}
 G.A. Margulis,  A. Nevo,  E.M. Stein, Analogs of Wiener's ergodic theorems for semisimple Lie groups. II. Duke Math. J. 103 (2000), 233--259.




\bibitem{MW}
M. Maxwell,
M.B. Woodroofe,
Central limit theorems for additive functionals of Markov chains.
Ann. Probab. 28 (2000),  713–724.

\bibitem{MPU}
 F. Merlev$\grave{e}$de, M. Peligrad, S. Utev, Recent advances in invariance principles for stationary sequences. Probab. Surv. 3 (2006), 1–36.

\bibitem{MP}
E. Mielkaitis, V. Paulauskas,
Rates of convergence in the CLT for linear random fields. Lith. Math. J. 51 (2011),  233–250.

\bibitem{N}
A. Nevo,   Pointwise ergodic theorems for actions of groups. Handbook of dynamical
systems  1B (2006)   871--982.



\bibitem{PZ}
M. Peligrad, N. Zhang,
On the normal approximation for random fields via martingale methods.
Stochastic Process. Appl. 128 (2018),  1333–1346

\bibitem{P}
V. V. Petrov,
Sums of Independent Random Variables.
Springer-Verlag, Berlin and Heidelberg,
(2011).

\bibitem{Pr}
Yu. V. Prokhorov, Convergence of Random Processes and Limit Theorems in Probability Theory,
Teor. Veroyatnost. i Primenen., (1956), Volume 1, Issue 2, Pages 177–238.

\bibitem{S}
 G. Samorodnitsky,
 Stochastic processes and long range dependence,
 Springer, Cham ( 2016).

\bibitem{St}
M.L. Straf,
Weak convergence of stochastic processes with several parameters,
Proc. 6th Berkeley Symp.,v.2 (1972) 187--221.

\bibitem{SS}
 W. Stute and G. Schuman, A General Glivenko-Cantelli Theorem
 for Stationary Sequences of Random Observations. Scand. J. Statist. 7 (1980), 102-104


 \bibitem{Te}
 A. Tempelman, Ergodic theorems for group actions
 Kluwer, Dordrecht (1992).

 \bibitem{Tem}
 Tempelman, A. Randomized consistent statistical  inference  for random processes and fields. Statistical Inference for Stochastic Processes. v. 143,  89-105   (2022)

\bibitem{TS}

A.Tempelman, A. Shulman, Dominated and Pointwise  Ergodic Theorems with "weighted"  averages for bounded Lamperti Representations of Amenable  Groups.  J. Math.Anal.Appl. 474 (2019), 23--58.

 \bibitem {T1}
 A. Tempelman, Randomized multivariate  central limit theorems for  ergodic homogeneous  random fields,
Stochastic Processes and their Applications,
 143 (2022), 89-105.


\bibitem{To}
C. Tone,
A central limit theorem for multivariate strongly mixing random fields.
Probab. Math. Statist. 30 (2010)  215–222.

\bibitem{T}
H. G.Tucker,  A generalization of the Glivenko-Cantelli theorem. Ann. Math. Statist. 30 (1959), 828–830.

\bibitem{VW}
D. Volný, Y. Wang, An invariance principle for stationary random fields under Hannan's condition. — Stochastic Process. Appl. 124,  (2014), 4012–4029.

\bibitem{Y}

M. I. Yadrenko,  Spectral theory of random fields. Translated from the Russian. Translation Series in Mathematics and Engineering. Optimization Software, Inc., Publications Division, New York (1983).

\bibitem{Yu}
B. Yu,
Rates of Convergence for Empirical Processes of Stationary Mixing Sequences. Ann. Probab. 22(1), (1994), 94-116.


\bibitem{ZRP}
N. Zhang, L. Reding, M. Peligrad,
On the quenched central limit theorem for stationary random fields under projective criteria.
J. Theoret. Probab. 33(2020)  2351–2379.







 \end{thebibliography}
\end{document}